\documentclass[12pt,reqno]{amsart}
\usepackage{bm}
\usepackage{graphicx}
\usepackage{adjustbox}
\usepackage{amsthm,amssymb,amsmath}
\usepackage[T1]{fontenc}
\usepackage[utf8]{inputenc}
\usepackage[colorlinks=true, linkcolor=blue, urlcolor=blue, citecolor=blue]{hyperref}
\usepackage{xcolor}
\usepackage{mathrsfs}
\usepackage{enumitem}
\usepackage{setspace}
\usepackage{float}
\usepackage{mathtools}
\usepackage[all,cmtip]{xy}
\usepackage{tikz-cd}
\usepackage{relsize}
\usepackage{extarrows}

\tikzcdset{row sep/normal=50pt, column sep/normal=50pt}

\newtheorem{lemma}{Lemma}[section]
\newtheorem{theorem}[lemma]{Theorem}
\newtheorem{proposition}[lemma]{Proposition}
\newtheorem{corollary}[lemma]{Corollary}
\theoremstyle{definition}
\newtheorem{definition}[lemma]{Definition}
\newtheorem{remark}[lemma]{Remark}
\newtheorem{example}[lemma]{Example}
\newtheorem{convention}[lemma]{Convention}
\numberwithin{equation}{subsection}

\DeclareMathOperator{\Gr}{Gr}
\DeclareMathOperator{\rk}{rk}

\DeclareMathOperator{\NE}{NE}
\DeclareMathOperator{\Nef}{Nef}

\newcommand{\Eff}{\mathrm{Eff}}

\newcommand{\R}{\mathbb{R}}

\newcommand{\cO}{\mathcal{O}}

\newcommand{\Aut}{\text{\rm Aut}}
\newcommand{\Gal}{\text{\rm Gal}}

\newcommand{\C}{\mathbb{C}}
\newcommand{\e}{E_{\ast}}

\newcommand{\p}{\widetilde{p}}

\newcommand{\Optwoo}{\mathcal{O}_{\mathrm{Gr}(r_{2}, E_{2\ast})}(1)}
\newcommand{\Oponee}{\mathcal{O}_{\mathrm{Gr}(r_{1}, E_{1\ast})}(1)}

\newcommand{\w}{\widetilde}
\newcommand{\grt}{{\textrm{Gr}}(r, \widetilde{E})}

\newcommand{\grp}{{\textrm{Gr}}(r, E_{\ast})}
\newcommand{\grpp}{{\emph{Gr}}(r, E_{\ast})}
\newcommand{\grtq}{{\textrm{Gr}}(q, \widetilde{E})}

\newcommand{\pS}{\Gr(r_1, E_1{_*})\,\times_X\, \Gr(r_2,E_2{_*})}
\newcommand{\wS}{\Gr(r_1,\widetilde E_1)\,\times_Y\, \Gr(r_2,\widetilde E_2)}
\newcommand{\we}{\Gr(r_1,\widetilde E_1)}
\newcommand{\wee}{\Gr(r_2,\widetilde E_2)}
\newcommand{\pae}{\Gr(r_1, E_1{_{\ast}})}
\newcommand{\paee}{\Gr(r_2, E_2{_{\ast}})}

\title{Positive Cones of Parabolic Grassmann Bundle over a curve}
\author{Ashima Bansal and Shivam Vats}
\date{}



\begin{document}

\begin{abstract}
In this article, we define the parabolic Grassmann bundle associated to a parabolic vector bundle over a smooth projective variety, generalizing the construction of parabolic projective bundles developed in \cite{BL}. We determine its Néron--Severi group and compute its nef, pseudoeffective, and Mori cones over smooth projective curves. We also compute the corresponding cones for the fiber product of two parabolic Grassmann bundles over a smooth projective curve.
\end{abstract}

\maketitle

\section{Introduction}
The study of cones of divisors has become a fundamental topic in algebraic geometry. For a smooth irreducible complex projective variety $X$, two of the most important numerical invariants are the nef cone $\textrm{Nef}^{\,1}(X)$ and the pseudoeffective cone $\overline{\textrm{Eff}}^{1}(X)$. These cones encode significant information about the geometry of $X$ and have played a central role in the development of modern birational geometry. Moreover, they have important applications to positivity questions, interpolation problems, and Seshadri constants.

Miyaoka initiated the study of nef and pseudoeffective divisors on the projectivization $\mathbb{P}(E)$ of a vector bundle $E$ over a smooth projective curve \cite{Mi}. More generally, Fulger described the cone of effective $k$-cycles on $\mathbb{P}(E)$ in terms of the numerical data arising from the Harder--Narasimhan filtration of $E$  \cite{Ful}. In particular, his results recover and generalize Miyaoka's description of the nef and pseudoeffective cones of divisors.

Let $E$ be a vector bundle on an irreducible smooth projective curve $X$ defined over $\mathbb{C}$. The Grassmann bundle $\Gr(r,E)$, parametrizing all $r$-dimensional quotients of the fibres of $E$, provides a natural generalization of the projective bundle $\mathbb{P}(E)$. The nef cone and the pseudoeffective cone of $\Gr(r,E)$, for $1\,\le\, r\,\le \,\operatorname{rk}(E)-1,$ were subsequently computed in \cite{BP} and \cite{BHP}, respectively. 

More recently, in \cite{BHNN}, the authors  computed the Mori cone of $\Gr(r,E)$, for suitably chosen $r$. As an application they studied the Seshadri constants of ample line bundles on $\Gr(r,E)$. Under certain numerical assumptions on the Harder--Narasimhan filtration of $E$, they obtained explicit descriptions of these Seshadri constants.

The study of positivity properties was further extended in \cite{MR} to Grassmann bundles over higher-dimensional smooth projective varieties. Moreover, the authors determined the nef and pseudoeffective cones of the fiber product of two Grassmann bundles over a smooth projective curve.

Mehta and Seshadri introduced the notion of a parabolic bundle on a Reimann surface \cite{MS}. Maruyama and Yokogawa generalized it to higher dimensional varieties \cite{MY}. Then Biswas established a two-way relationship between the parabolic vector bundles and the orbifold vector bundles \cite{Bi1}. Given a parabolic vector bundle $E_{\ast}$, Biswas and Laytimi defined $\mathbb{P}(E_{\ast})$ in \cite{BL}, a parabolic analog of the projective bundle. This is done using the notion of ramified principal bundles (see \cite{BBN} and \cite{Bi2} for ramified principal bundles). They also construct the tautological line bundle $\mathcal{O}_{\mathbb{P}(E_{\ast})}(1)$ on $\mathbb{P}(E_{\ast})$. These provide important examples of finite group quotients which arise as finite quotients of the projectivization of orbifold bundles. Using the quotient description of parabolic projective bundles, the numerical groups of higher cycles were determined and explicit descriptions of the pseudoeffective and nef cones of $k$--cycles on parabolic projective bundles over smooth complex projective curves were obtained in \cite{BBM2}.

In view of these developments, it is natural to study the parabolic analogue of the Grassmann bundle and investigate its positivity properties. In Section~3, we generalize the construction in \cite{BL} to define parabolic Grassmann bundle $\textrm{Gr}(r, \e)$, for $r\,\in\,[1, \,\textrm{rank}(\e)-1]$, together with their associated tautological line bundle. As in the case of parabolic projective bundles, these varieties also arise as finite quotients of the Grassmann bundles associated to the corresponding orbifold bundles. Unlike classical Grassmann bundles, however, parabolic Grassmann bundles need not be smooth (see, Example~\ref{example}). Consequently, the quotient morphism from the corresponding orbifold Grassmann bundle need not be flat, so the usual flat pullback of cycles is therefore not available. Nevertheless, since parabolic Grassmann bundles arise as finite quotients of smooth varieties, we apply Fulton's notion of special pullbacks for finite quotient morphisms \cite{Fu}.

Our first main result provides an explicit description of the Néron--Severi group of a parabolic Grassmann bundle over a smooth projective curve,  Proposition \ref{prop:N1}.
Using only the numerical data associated to the Harder--Narasimhan filtration of $\e$, we compute the a rational number $\theta_{\e, r}$ (see \eqref{thetap}), where $r\,\in\, [1, \textrm{rank}(\e)\,-\,1]$. The Lemma \ref{positivity}, shows that $\theta_{\e ,r}$ controls the positivity of the tautological line bundle $\mathcal{O}_{\textrm{Gr}(r, \e)}(1)$. Using this $\theta_{\e, r}$, we compute the Nef cone of $\textrm{Gr}(r, \e)$, Theorem \ref{thm:nef-cone}.

Further, in Section~3, we determine the pseudoeffective cone of the parabolic Grassmann bundle \textrm{Gr}(r, $E_{\ast}$). In contrast to the description obtained in \cite{BHNN}, we establish another description of the Mori cone of the Grassmann bundle \textrm{Gr}($q, \widetilde E$) over a smooth projective curve, under a suitable assumption on the Harder--Narasimhan filtration of $\widetilde E$; see Lemma~\ref{lem:orbifold-mori}.
This description plays a crucial role in our computation of the Mori cone of the parabolic Grassmann bundle \textrm{Gr}$(q,E_{\ast})$, which is carried out in Theorem~\ref{thm:mori-parabolic}.

In Section~4, we study the nef and pseudoeffective cones of the fiber product of two parabolic Grassmann bundles over a smooth complex projective curve $X$. Via the correspondence between parabolic and orbifold bundles, this variety arises as a finite quotient of the fiber product of the corresponding Grassmann bundles equipped with the diagonal action of a finite group. Let $E_{1_{\ast}}$ and $E_{2_{\ast}}$ be parabolic bundles on $X$, and let $\widetilde E_1$ and $\widetilde E_2$ denote the corresponding orbifold bundles. Further we compute the Mori cone of the fiber product $\textrm{Gr}(r_1, \widetilde {E}_1)\,\times\,
\textrm{Gr}(r_2, \widetilde E_2),$ for suitably chosen integers $r_1$ and $r_2$ (Proposition~\ref{morif}).  As an application, we determine the Mori cone of the fiber product of the corresponding parabolic Grassmann bundles; see Theorem~\ref{morifp}.

\section{Preliminaries}
This section recalls Chow groups, numerical groups, the associated closed cones, and the theory of orbifold and parabolic bundles that will be used throughout the paper (see \cite{Fu,FL2, Bi1} for further details).

\subsection{Chow groups and Numerical groups}

Throughout this subsection, $X$ is a complex projective variety of dimension $n$ not necessary smooth.

\subsubsection{Cycles and Chow groups}

Let $Z_{k}(X)$ denote the group of $k$--cycles on $X$ with coefficients in $\mathbb R$. Any subscheme $Z\,
\subset\, 
X$ of dimension $k$ has a fundamental cycle denoted by $[Z]\,\in\,{Z}_{k}(X)$ defined as in \cite[\S~1.5]{Fu}. 
To study the geometry of cycles on $X$, various equivalence relations have been 
introduced on $Z_{k}(X)$. One example is rational equivalence; see \cite[\S~1.3 and \S~1.6]{Fu}.
The Chow group $CH_{k}(X)$ is the quotient of ${Z}_{k}(X)$ modulo rational equivalence, which may
still have infinite rank. When $X$ is smooth, denote $CH^{k}(X) \,=\, CH_{n-k}(X)$.
There is a graded ring structure on $CH^{\ast}(X)\,:=\, \bigoplus_{k\geq 0} CH^{k}(X)$ \cite[\S~8]{Fu}. 

Take a vector bundle $E$ on $X$. For every $k$ and $m$, there exists a linear map
$$CH_{m}(X)\ \, \xrightarrow{\,\,\,c_{k}(E)\,\cap\,_{-}\,\,\,}\ \,CH_{m-k}(X),$$ 
where $c_{k}(E)$ denote the $k$-th Chern class \cite[\S~3.2]{Fu}.

Let 
$
A^{\ast}(X)\,=\,\bigoplus_{k\geq 0} A^{k}(X)
$ denote the operational Chow ring of $X$ as defined in
\cite[Chapter 17, Definition 17.3]{Fu}. Throughout this paper, we work with operational Chow groups with real coefficients; that is, 
$
A^{k}(X)\,:=\,A^{k}(X)_{\mathbb R}
\,=\,
A^{k}(X)\,\otimes_{\mathbb Z}\,\mathbb R.
$ The operational Chow ring $A^{\ast}(X)$ is an associative graded ring with unity $1$, whose degree $k$ component is the $k$--th operational Chow group $A^{k}(X)$. If $X$ is a smooth projective variety, then there is a canonical isomorphism $
A^{\ast}(X)\,\cong\,CH^{\ast}(X).$

\subsubsection{Chow ring of a finite group quotient}
Let $G$ be a finite group acting on a smooth projective variety $X$, and let $Y\,=\,X/G$ be the quotient variety. Denote by $\pi\,:\,X\,\longrightarrow\,Y$ the finite quotient map.
Let ${CH}_{\ast}(X)^{G}$ denote the ring of $G$-invariants of ${CH}_{\ast}(X)$. Then there is a canonical group isomorphism ${CH}_{\ast}(Y)\, =\, {CH}_{\ast}(X)^{G}$.
For any subvariety $W$ of $X$, let 
\begin{equation*}
I_{W}\ =\ \{g\,\in\,G\,\,\big\vert\,\, g_{|_W}\,=\, id_{W}\}    
\end{equation*}
be the inertia group, and let 
\begin{equation*}
e_{W}\,=\, \textrm{card}(I_{W})/\textrm{deg}_{i}(W/V),    
\end{equation*}
where $V\,=\, \pi(W)$, and ${\rm deg}_{i}(W/V)$ is the degree of inseparability of $K(W)$ over $K(V)$, the function fields of $W$ and $V$, respectively. We recall from \cite[Example 1.7.6]{Fu} that, for a subvariety $V$ of $Y$, the special pullback
\begin{equation}\label{special}
\pi_{\textrm{sp}}^{\ast}[V]\,=\, \sum\, e_{W} [W],    
\end{equation}
where the sum over all irreducible components $W$ of $\pi^{-1}(V)$. This determines an isomorphism $Z_{\ast}(Y)\,=\, Z_{\ast}(X)^{G}$, and ${CH}_{\ast}(X)^{G}$ is the quotient of $Z_{\ast}(X)^{G}$ modulo the subspace generated by
\begin{equation*}
\left\{\sum_{g\,\in\,G}g_{\ast}[\textrm{div}(r)]\,\,\big\vert\,\,r\,
\in\,K(W)^{\ast},\,\,\, W\,\subset\,X\right\}.    
\end{equation*}
Note that the composition of maps
\begin{equation}\label{comp}
{CH}_{\ast}(Y)\,\xlongrightarrow{\,\,\pi_{\textrm{sp}}^{\ast}\,\,}\, {CH}_{\ast}(X)^{G}\,
\hookrightarrow\,{CH}_{\ast}(X)\,\xlongrightarrow{\,\,\pi_{\ast}\,\,}\, {CH}_{\ast}(Y)     
\end{equation}
is multiplication by card($G$). In addition, ${CH}_{\ast}(Y)$ may also be made into a ring. Indeed, in
this case, one has an isomorphism
\begin{equation*}
 {CH}_{\ast}(Y)\ =\ {\rm CH}_{\ast}(X)^{G},   
\end{equation*}
so ${CH}_{\ast}(Y)$ is the ring of $G$-invariants of ${CH}_{\ast}(X)$.

In fact, if $V,W$ are subvarieties of $Y$, one may construct a refined intersection class $V\cdot W$ in ${CH}_{\ast}(V\,\cap\,W)$ defined as 
\begin{equation*}
V\cdot W \ =\ (1/|G|)\eta_{\ast}(\pi_{\textrm{sp}}^{\ast}[V]\cdot \pi_{\textrm{sp}}^{\ast}[W]),   
\end{equation*}
where $\eta$ is the projection from $\pi^{-1}(V\,\cap\,W)$ to $V\,\cap\, W$. Note that $\pi_{\textrm{sp}}^{\ast}(a\cdot b)\,=\, \pi_{\textrm{sp}}^{\ast}(a)\cdot \pi_{\textrm{sp}}^{\ast}(b)$,\,\, \,\textrm{and}\,\,\,$\pi_{\ast}(\pi_{\textrm{sp}}^{\ast}(a)\cdot c)\,=\, a\cdot \pi_{\ast}(c)$, for cycles $a,b$ on $Y$, and cycle $c$ on $X$.

The canonical homomorphism 
\begin{equation*}
A^{\ast}(Y)\,\xlongrightarrow{\cap\,[Y]}\, CH_{\ast}(Y)    
\end{equation*}
is an isomorphism of rings. This shows in particular that the ring structure on $CH_{\ast}(Y)$ is independent of $X$, so produces a  pull-back homomorphism for arbitrary morphisms of such varieties. See \cite[Examples 8.3.12 and 17.4.10]{Fu} for more details.
\subsubsection{Numerical equivalence}
To define numerical equivalence, we work with an equivalence relation that is coarser than rational equivalence.

\textbf{Smooth case.}
When $X$ is smooth, there is an intersection pairing
\begin{equation}\label{e1}
CH_{k}(X)\,\times\, CH^{k}(X)
\,\longrightarrow\,
\mathbb{R}
\end{equation}
determined by the ring structure on $CH^{\ast}(X)$ and the natural point counting degree function $\deg\,:\,CH_{0}(X)\,\longrightarrow\, \mathbb{R}.$ The \textit{numerical group} $N_{k}(X)$ is defined as the quotient of $CH_{k}(X)$ by the kernel of this pairing; denote $N^{k}(X)\,:=\,N_{n-k}(X)$. The pairing in \eqref{e1} induces a perfect pairing $N_{k}(X)\,\times\, N^{k}(X)\,\longrightarrow\, \mathbb{R},$ in particular, we have $N^{k}(X)\,\cong\, (N_{k}(X))^{\vee}.$

\medskip

\textbf{Singular case.}
When $X$ is singular, we do not have an intersection pairing. Instead, the Chern class action can be used. A $k$--cycle
$Z$ is called \emph{numerically trivial} --- denoted by $Z\,\equiv\, 0$ --- if
\[
\deg(P\,\cap\, [Z]_{\mathrm{Chow}})\,=\,0
\]
for any weight $k$ polynomial $P$ in Chern classes of vector bundles on $X$.
The \textit{numerical group} is then defined by
\[
N_{k}(X)\,:=\,CH_{k}(X)/\equiv.
\]
It is a real vector space of finite dimension, and it is nonzero only when $0\,\leq\, k\,\leq\, \textrm{dim}\,X\,=\, n.$ The class in $N_{k}(X)$ of a real $k$--cycle $Z$ is denoted by $[Z]$. Clearly, both
$N_{0}(X)$ and $N_{n}(X)$ are isomorphic to $\mathbb{R}$.

The \textit{dual numerical group} is defined by $N^{k}(X)\,:=\,(N_{k}(X))^{\vee}$ is no longer isomorphic to $N_{n-k}(X)$. Equivalently, it can be defined as follows:

\begin{equation}\label{upper numerical groups}
{N}^{k}(X)\ = \ \frac{\text{Homogeneous Chern }\, \mathbb{R}\text{--polynomials $P$ of
weight }\, k}{\text{Chern polynomials }\, P\, \text{ such that }\, P\,\cap\,
\alpha \,=\,0\, \text{ for all}\, \alpha\,\in\, N_{k}(X)}.
\end{equation}
The multiplication of Chern polynomials endows $N^{\ast}(X)$ with the structure of a graded ring.
Moreover, the action of Chern classes induces maps
\[
N^{k}(X)\times N_{m}(X)
\longrightarrow
N_{m-k}(X),
\qquad
(P,\alpha)\longmapsto P\cap\alpha,
\]
or simply,
\[
P\cdot\alpha.
\]
The action on the fundamental class $[X]$ yields a natural map
\begin{equation}\label{cyclefication}
\phi\,:\,
N^{k}(X)
\,\longrightarrow\,
N_{n-k}(X),
\qquad
P\,\longmapsto\, P\,\cap\, [X],
\end{equation}
which is, in general, not an isomorphism.
However, for $k\,=\,1$, the map $\phi\,:\,N^{1}(X)\,\longrightarrow\, N_{n-1}(X)$ is injective; see \cite[Example 19.3.3]{Fu}. Dually, $N^{n-1}(X)\,\longrightarrow\, N_{1}(X)$ is onto. More generally, $N_{\ast}(X)$ is a module over $N^{\ast}(X)$.

For notational convenience, whenever no confusion is likely to arise, we shall denote
\[
P\,\cap\,[X]\,\in\, N_{n-k}(X)
\]
simply by $P$.

\begin{remark}\cite[Example~2.1]{FL2}.
The space $N^{1}(X)$ is the real vector space of first Chern classes of invertible sheaves modulo those having vanishing intersection against every curve. Throughout, we identify an invertible sheaf with the corresponding rational equivalence class of Cartier divisors. 
\end{remark}

\begin{remark}\label{pdecent}
The relation of numerical equivalence is coarser than rational equivalence. Hence the natural quotient map $CH_{\ast}(X)\,\longrightarrow\, N_{\ast}(X)$ is well defined.

Furthermore, if
\[
\pi\,:\,X\,\longrightarrow\, Y
\]
is a proper morphism of projective varieties, then the proper pushforward on Chow groups descends to numerical groups, giving a commutative diagram
\begin{equation}\label{decent}
\begin{tikzcd}
CH_{\ast}(X) \arrow[r] \arrow[d,"\pi_{\ast}"']
&
N_{\ast}(X) \arrow[d,"\pi_{\ast}"]
\\
CH_{\ast}(Y) \arrow[r]
&
N_{\ast}(Y).
\end{tikzcd}
\end{equation}
\end{remark}

\subsection{Positive cones}

\subsubsection{The pseudoeffective cone}

A say that a class $\alpha\,\in\, N_{k}(X)$ is \emph{effective} if $\alpha\,=\,[Z]$ for some effective cycle $Z$. This notion is closed under positive linear combinations, and hence it is natural to consider the following:

\begin{definition}
The closure of the convex cone generated by effective $k$--cycles in $N_{k}(X)$ is denoted by $\overline{\Eff}_{k}(X).$ It is called the \emph{pseudoeffective cone}. A class $\alpha\,\in\, N_{k}(X)$ is called \emph{pseudoeffective} if $\alpha\,\in\,\overline{\Eff}_{k}(X).$ 
\end{definition}

Note that $\overline{\Eff}_{1}(X)$ is referred to as the \emph{Mori cone} (or the \emph{closed cone of curves}), which is also denoted by $\overline{\NE}(X)$ in the literature.

\begin{definition}
A class $\beta\,\in\, N^{k}(X)$ is called \emph{pseudoeffective} if $\phi(\beta)\in \overline{\Eff}_{n-k}(X),$ where $\phi$ is the map in \eqref{cyclefication}. The pseudoeffective dual classes form a closed cone in $N^{k}(X)$, and denoted by $\overline{\Eff}^{\,k}(X).$
\end{definition}

\subsubsection{The nef cone}

Recall the perfect pairing  $N^{k}(X)\,\times\, N_{k}(X)
\,\longrightarrow\,
\mathbb{R}$ defined by $(P,\alpha)\,\longmapsto\, P\,\cap\,\alpha.$

\begin{definition}
The \emph{nef cone} $\Nef^{\,k}(X)\,\subset\, N^{k}(X)$ is the dual cone of $\overline{\Eff}_{k}(X)\,\subset\, N_{k}(X)$ with respect to the above pairing.
\end{definition}
By definition, nefness is preserved under proper pullbacks. If $X$ is smooth, a cycle $\alpha\,\in\,N^{k}(X)$ is nef if and only if $(\alpha\cdot \beta)\,\geq\, 0$ for all effective cyles $\beta$ of dimension $k$.

For the purposes of this article, we only require $\textrm{Nef}^{1}(X)$. For simplicity, we denote $\textrm{Nef}^{1}(X)$ by $\textrm{Nef}(X)$.

\subsection{Orbifold bundle}
Let $Y$ be a smooth complex projective variety of dimension $n$. Its group of algebraic automorphisms will be 
denoted by $\textrm{Aut}(Y).$ Let $\Gamma\, \subset\, \textrm{Aut}(Y)$ be a finite subgroup. So $\Gamma$
acts on $Y$ through algebraic automorphisms.

An \textit{orbifold bundle} on $Y$, with $\Gamma$ as the \textit{orbifold group}, is a vector bundle $V$ on 
$Y$ together with a lift of the action of $\Gamma$ on $Y$ to $V$, i.e., $\Gamma$ acts on the total space of 
$V$ such that the action of any $g\,\in \,\Gamma$ gives a vector bundle isomorphism between $V$ to 
$(g^{-1})^{\ast}V$. A subsheaf $F$ of an orbifold bundle $V$ is called an \textit{orbifold subsheaf} if 
the action of $\Gamma$ on $V$ preserves $F$.

Fix a polarization on $Y$ that is preserved by the action of $\Gamma$.
An orbifold bundle $V$ on $Y$ is called \textit{orbifold semistable} (respectively, \textit{orbifold 
stable}) if for any orbifold subsheaf $F$ of $V$, with $0\,<\,\textrm{rank(F)}\,<\,\textrm{rank(V)}$, the
following inequality holds:
\begin{equation*}
\frac{\textrm{deg}(F)}{\textrm{rank} (F)} \,\leq\, \frac{\textrm{deg} (V)}{\textrm{rank} (V)}\,\, \
\left(\textrm{respectively,}\,\, \frac{\textrm{deg} (F)}{\textrm{rank} (F)} \,<\,
\frac{\textrm{deg} (V)}{\textrm{rank} (V)}\right).
\end{equation*}

\begin{proposition}[{\cite[Lemma 2.7]{Bi1}}]\label{orbss}
An orbifold bundle $V$ is orbifold semistable if and only if it is semistable in the usual sense.
\end{proposition}
\subsection{Parabolic vector bundles.} 
Let $X$ be an irreducible smooth projective variety defined over $\C$. A \textit{parabolic vector bundle} on $X$ with a divisor $D$ is a vector bundle $E$ on $X$ together with filtrations of its restrictions to the components of $D$, each equipped with strictly increasing weights in $[0,1)$ (see \cite{MY} and \cite{Se}).

The notions of subsheaf, quotient, direct sum, tensor product, dual, symmetric product, and exterior powers for vector bundles extend naturally to parabolic vector bundles. In addition, we have the notions of semistability 
and stability for parabolic vector bundles (see \cite{MY}, \cite{Bi2}, \cite{Yo}).

Every parabolic sheaf $\e$ has the \textit{Harder-Narasimhan filtration}, i.e., there exists a unique filtration 
\begin{equation}\label{HN}
E\, =\, E^{m}\,\supsetneq\, E^{m-1}\, \supsetneq\, \cdots\,\supsetneq\, E^{0}\, = \,0
\end{equation}
such that all $(E^{i}/E^{i-1})_{\ast}$ with the induced quotient
parabolic structure are parabolic semistable and \textrm{par-}$\mu((E^{i}/E^{i-1})_{\ast})\, >\, $ \textrm{par-}$\mu((E^{i+1}/E^{i})_{\ast})$ for all $i\,\in\, \{1,\, 2,\, \cdots,\, m\}$. Let
\begin{equation*}
Q^{i}_\ast\ \,:=\, \ \left( \frac{E^{i}}{E^{i-1}} \right)_{\ast}
\textnormal{be the quotient parabolic sheaf}.
\end{equation*}
Let
\begin{equation}\label{parrank}
r_i\, =\, \operatorname{rk}(Q^{i}_{\ast}), \quad d_i = \textrm{par-}\textrm{deg} \,(Q^{i}_{\ast}),\quad \mu_i\, :=\, \textrm{par-}\mu(Q^{i}_{\ast})\, =\, \frac{d_i}{r_i}
\end{equation}
be the rank, parabolic degree and parabolic slope of $Q^{i}_{\ast}$, respectively.
Hence, we have
\begin{equation*}
\mu_{m}\,<\, \mu_{2}\,<\, \cdots < \mu_{1}.
\end{equation*}
\subsection{Correspondence between  orbifold bundles and parabolic bundles }

Let $Y$ be a smooth projective variety equipped with a faithful action of a finite group $\Gamma$, and let
\begin{equation*}
p\ :\ Y \ \longrightarrow \ Y/\Gamma\ \,=:\,\ X
\end{equation*}
be the quotient map, such that $X$ is smooth. As described in \cite{Bi1}, there is an equivalence of categories between parabolic vector bundles on $X$ (with rational weights determined by the ramification data of $p$) and $\Gamma$--equivariant vector bundles on $Y$. Under this correspondence, a parabolic bundle $\e$ on $X$ corresponds to a $\Gamma$--bundle $\widetilde{E}$ on $Y$, and we have:
\begin{equation*}
\textrm{deg}(\widetilde{E})\, =\, |\Gamma|. \textrm{par-}\textrm{deg}(\e) \qquad \text{ and } \qquad \textrm{rk}(E_{\ast})\, =\, \textrm{rk}(\widetilde{E}), 
\end{equation*}
where $|\Gamma|$ denotes the cardinality of the finite group $\Gamma$.

\begin{proposition}[{\cite[Lemma 3.16]{Bi1}}]\label{semistable} The orbifold bundle
$\widetilde{E}$ is orbifold semistable if and only if $\e$ is parabolic semistable.
\end{proposition}

\begin{remark}\label{allss}
By Proposition \ref{orbss} and Proposition \ref{semistable} we have $\w{E}$ is semistable in the usual sense if and only if $\e$ is parabolic semistable. 
\end{remark}

In fact, the above correspondence between orbifold vector bundles and parabolic vector bundles preserves the Harder--Narasimhan filtration. Consider the Harder-Narasimhan filtration of $\e$ defined in \eqref{HN}, 
\begin{equation*}
E^m\, \supset\, E^{m-1}\, \supset\, \cdots\, \supset\, E^0\,=\,0
\end{equation*}
such that $(E_{i}/E_{i-1})_{\ast}$ is parabolic semistable for each $i\,\in\,[1, m]$, and parabolic slope 
\begin{equation*}
\textrm{par-}\mu(E_{i}/E_{i-1})\,>\, \textrm{par-}\mu(E_{i+1}/E_{i}).
\end{equation*}
Let $\widetilde{E}$ be the orbifold bundle corresponding to the parabolic bundle $E_{\ast}$. Correspondingly, $\widetilde{E}$ admits the Harder--Narasimhan filtration:
\begin{equation}\label{hno}
\widetilde{E} \,=\, \widetilde{E}^{m} \,\supset\, \widetilde{E}^{m-1}\, \supset \cdots \supset\, \widetilde{E}^{0} \,=\, 0,
\end{equation}
such that each quotient $\widetilde{E}^{i}/\widetilde{E}^{i-1}$ is semistable and
\[
\mu\big(\widetilde{E}^{i}/\widetilde{E}^{i-1}\big)
\,>\,
\mu\big(\widetilde{E}^{i+1}/\widetilde{E}^{i}\big).
\]
\section{Parabolic Grassmann Bundles}
Let $X$ be an irreducible smooth projective variety defined over $\mathbb{C}$, and let $E$ be a vector bundle of rank $k$ on $X$. Let $E_*$ be a parabolic vector bundle on $X$ with parabolic divisor $D$ and underlying vector bundle $E$. Fix an integer $1\, \leq\, r\, \leq\, k\,-\,1.$

Let
\begin{equation}\label{rp}
\phi\, \colon\, E_{\textrm{GL}(k,\C)} \,\longrightarrow\, X
\end{equation}
be the ramified principal $\textrm{GL}(k,\C)$--bundle corresponding to $E_*$ (see \cite{BBN}). Let $\Gr(r,\C^k)$ be the Grassmannian variety parametrizing all $r$-dimensional quotients of $\C^k$. The standard action of $\textrm{GL}(k,\C)$ on $\C^k$ produces an action of $\Gr(k, \mathbb{C})$ on $\Gr(r,\C^k)$. Let
\begin{equation}\label{action}
f \,\colon\, \textrm{GL}(k,\C) \,\longrightarrow\, \Aut(\Gr(r,\C^k))
\end{equation}
be the corresponding homomorphism.

\begin{definition}\label{pargras}
The \emph{parabolic Grassmann bundle} associated to the parabolic vector bundle $E_*$, denoted by $\Gr(r,E_*)$, is defined to be the associated (ramified) fiber bundle
\begin{equation}\label{pargrass}
\pi_r \,\colon\, \Gr(r,E_*) \,:=\, E_{\textrm{GL}(k,\C)} \,\times^{\textrm{GL}(k,\C)}\, \Gr(r,\C^k) \,\longrightarrow\, X.
\end{equation}
\end{definition}
In other words, $\Gr(r, E_{\ast})$ is the quotient of $E_{\textrm{GL}(k, \C)}\,\times \Gr(r, \C^{k})$ where two points $(z_{1},b_{1})$
and $(z_{2},b_{2})$, where $z_{1}, z_{2}\,\in\, E_{\Gr(k, \C)}$ and $b_{1}, b_{2}\,\in\,\Gr(r,\C^k)$, are identified if there is $A\,\in\, \textrm{GL}(k, \C)$ such that 
$z_{2}\,=\, z_{1}A$ and $b_{2}\,=\, A^{-1}(b_{1})$.

We now construct the tautological line bundle on $\Gr(r,E_*)$. 

Take a point $x \,\in\, D$; here $x$ need not be a smooth point of $X$. Let $z \,\in\, \phi^{-1}(x)$, where $\phi$ is the morphism in \eqref{rp}. Let $G_z \,\subset\, \textrm{GL}(k,\C)$ be the isotropy subgroup for the action of $\textrm{GL}(k,\C)$ on $E_{\textrm{GL}(k,\C)}$. We recall that $G_z$ is a finite group, see \cite{BL}. Let $n_x$ denote the order of $G_z$ (it is independent of the choice of $z$ because $\textrm{GL}(k,\C)$ acts transitively on the fiber $\phi^{-1}(x)$). The number of distinct integers $n_{x}$ as $x$ varies over $D$ is finite. Let
\begin{equation}\label{N}
N(E_*)\, :=\, \operatorname{l.c.m.}\{n_x  \, \,|\,\,x\,\in\, D\}
\end{equation}
be the least common multiple of all these integers $n_{x}$.
For simplicity, we write
\begin{equation}\label{n}
\lambda \,:=\, N(E_*).
\end{equation}

For any point $y\,\in\, \Gr(r,\C^k)$, let
\begin{equation}\label{H}
H_y\, \subset\, \textrm{GL}(k,\C)
\end{equation}
be the isotropy subgroup for the natural action of $\textrm{GL}(k, \C)$ on $\Gr(r,\C^k)$; 
defined by \eqref{action}. Then $H_y$ is a maximal parabolic subgroup of $\textrm{GL}(k,\C)$ (see \cite{Br}). The group $H_{y}$ acts on the fiber of the tautological line bundle bundle $\cO_{\Gr(r,\C^k)}(1)\, \longrightarrow\, \Gr(r,\C^k)$ over the point $y$. From the definition of $\lambda$ in \eqref{N} it follows immediately that for any $z \,\in\, \phi^{-1}(D)$, and any $y \,\in\, \Gr(r,\C^k)$, the finite group $G_z \,\cap\, H_y\,\subset\, \textrm{GL}(k, \mathbb{C})$ acts trivially on the fiber of the line bundle
\[
\cO_{\Gr(r,\C^k)}(\lambda)
\,:=\,
\cO_{\Gr(r,\C^k)}(1)^{\otimes \lambda}
\]
over the point $y$. Consider the action of $\textrm{GL}(k,\C)$ on the total space of the line bundle $\cO_{\Gr(r,\C^k)}(\lambda)$ constructed using the standard action of $\textrm{GL}(k,\C)$ on $\C^k$, and let
\begin{equation}\label{fb}
E_{\textrm{GL}(k,\C)}(\cO_{\Gr(r,\C^k)}(\lambda))\,:=\,
E_{\textrm{GL}(k,\C)} \,\times^{\textrm{GL}(k,\C)}\, \cO_{\Gr(r,\C^k)}(\lambda)
\,\longrightarrow\, X
\end{equation}
be the associated fiber bundle. Since the natural projection $\cO_{\Gr(r,\C^k)}(\lambda) \,\longrightarrow\, \Gr(r,\C^k)$ is $\textrm{GL}(k,\C)$--equivariant, it induces a projection 
\begin{equation}\label{taut}
E_{\textrm{GL}(k,\C)}(\cO_{\Gr(r,\C^k)}(\lambda))
\,\longrightarrow\,
\Gr(r,E_*),
\end{equation}
over the map $p\,:\,Y\,\longrightarrow\,X$.

Using the observation above that $G_z \,\cap\, H_y$ acts trivially on the fiber of $\cO_{\Gr(r,\C^k)}(\lambda)$ over $y$ it follows immediately that $E_{\textrm{GL}(k,\C)}(\cO_{\Gr(r,\C^k)}(\lambda))$ in \eqref{taut} is an algebraic line bundle over the variety $\Gr(r,E_*)$.

\begin{definition}
The line bundle
\begin{equation}\label{tauto}
E_{\textrm{GL}(k,\C)}(\cO_{\Gr(r,\C^k)}(\lambda))
\,\longrightarrow\,
\Gr(r,E_*)
\end{equation}
will be called the \emph{tautological line bundle} on $\Gr(r,E_*)$, and it will be denoted by $\cO_{\Gr(r,E_*)}(1).$

\end{definition}
Let $X$ be a smooth complex projective curve with $D$ being a reduced effective divisor on it. Let $E_{\ast}$ be a parabolic bundle on $X$ with parabolic structure on $D$ such that all the parabolic weights are rationals. Assume that $E_{*}$ satisfy \cite[Assumption~3.2]{Bi1}.  The ``covering lemma'' says that there exists a smooth projective curve $Y$ and a finite Galois morphism
\[
p\, \colon\, Y\, \longrightarrow\, X
\]
with Galois group
\[
\Gamma \,=\, \Gal(K(Y)/K(X)),
\]
such that $E_*$ corresponds to a unique orbifold vector bundle $\widetilde{E}$ on $Y$ (see Section~3 of \cite{Bi1}). The action of $\Gamma$ on $\w{E}$ produces a left action of $\Gamma$ on $\textrm{Gr}(r, \w{E})$. 
It can be noted, that the variety $\Gr(r, \e)$ in $\eqref{pargras}$ is the quotient 
$$\Gr(r, \w{E})/\Gamma\,=\, \Gr(r, \e).$$
We note that the quotient $\mathcal{O}_{\Gr(r, \w{E})}(\lambda)/\Gamma$ is a line bundle over $\grt/\Gamma$ because the isotropy subgroups, for the action of $\Gamma$ on $\grt$, act trivially on the corresponding fibers of $\mathcal{O}_{\grt}(\lambda)$. We have a natural isomorphism of the bundles 
\begin{equation}\label{cons}
\mathcal{O}_{\grt}(\lambda)/\Gamma\,=\,\mathcal{O}_{\grp}(1).
\end{equation}
Let
\[
\widetilde{\pi}_r \,\colon\, \Gr(r,\widetilde{E}) \,\longrightarrow\, Y
\]
be the Grassmann bundle over $Y$ parametrizing all $r$-dimensional quotients of the fibers of $\widetilde{E}$. Let $\w{p}_{r}$ be the quotient morphism
\[
\widetilde{p}_r \,\colon\, \Gr(r,\widetilde{E}) \,\longrightarrow\, \Gr(r,E_*).
\]
It follows from \eqref{cons} that
\begin{equation}\label{pullback}
 \w{p}_{r}^{\,\ast}(\mathcal{O}_{\grp}(1))\,=\, \mathcal{O}_{\grt}(\lambda).   
\end{equation}
Consider the following commuative diagram
\begin{equation}\label{0}
\begin{tikzcd}
\Gr(r,\widetilde{E}) \arrow[r, "\widetilde{\pi}_r"] \arrow[d, "\widetilde{p}_r"']
& Y \arrow[d, "p"] \\
\Gr(r,E_*) \arrow[r, "\pi_r"']
& X.
\end{tikzcd}
\end{equation}
\begin{remark}
In contrast to the usual Grassmann bundle $\textrm{Gr}(r, \widetilde{E})$, the bundle $\textrm{Gr}(r, {E}_{\ast})$ need not be smooth, as ramification may occur. We see this through the following example.
\end{remark}

\begin{example}\label{example}
Let \(Y\) be a smooth curve and let $E\,=\,\mathcal O_Y^{\oplus 2}\oplus \mathcal O_Y^{\oplus 2}.$ Write $E^{1}\,:=\,\mathcal O_Y^{\oplus 2},$ and $E^{2}\,:=\,\mathcal O_Y^{\oplus 2}.$ Let
\(\Gamma\,=\,\mathbb Z_{2}\,=\,\{1,\gamma\}\) act trivially on \(Y\) and on \(E\) it acts by
\[
\gamma\cdot(u,v)\,=\,(u,-v),
\qquad
u\in E^{1},\; v\in E^{2}.
\]
This induces an action of \(\Gamma\) on the  Grassmann bundle $\Gr(2,E)\,\longrightarrow\,Y.$ For \(y\,\in\,Y\), a point of \(\Gr(2,E)\) over \(y\) is a two-dimensional subspace $W\,\subset\, E_y\,=\,E_y^{1}\,\oplus\, E_y^{2}.$ Such a point is fixed by \(\gamma\) if and only if \(W\) is \(\gamma\)-invariant. Since \(\gamma\) acts by \(+1\) on \(E_y^{1}\) and by \(-1\) on \(E_y^{2}\), this is
equivalent to
\[
W\,=\,(W\cap E_y^{1})\,\oplus\, (W\,\cap\,E_y^{2}).
\]
Therefore the fixed locus is
\[
\Gr(2,E^{1})
\;\cup\;
\bigl(\Gr(1,E^{1})\times_Y \Gr(1,E^{2})\bigr)
\;\cup\;
\Gr(2,E^{2}).
\]
Since \(\operatorname{rk}(E^{1})\,=\,\operatorname{rk}(E^{2})\,=\,2\), we have
\[
\Gr(2,E^{1})\,\cong\, Y,
\qquad
\Gr(2,E^{2})\,\cong\, Y,
\]
and
\[
\Gr(1,E^{1})\,\times_Y\, \Gr(1,E^{2})
\]
has  dimension \(3\). On the other hand, $\Gr(2,E)\,\longrightarrow\, Y$ has dimension
$2(4\,-\,2)\,=\,5.$ Hence the components of the fixed locus have codimensions \(4,2,4\),
respectively. In particular, the fixed locus has codimension at least \(2\). So $\Gr(2,E)/\Gamma
$ is not smooth.
\end{example}

\subsection{Setup and Notations}\label{3.1}
From now on, let $X$ and $Y$ be smooth complex projective curves, and let ${E}_{\ast}$ be a parabolic bundle on $X$ of rank $k$ with parabolic divisor $D$, satisfies the additional hypotheses of \cite[Assumption~3.2]{Bi1}, and let $\widetilde{E}$ be the corresponding orbifold bundle. Let $1\,\leq\,r\,\leq\,k-1$, and consider the projection  
\[
\pi_r \,\colon\, \Gr(r,\,E_*)\, \longrightarrow\, X.
\]
Let $x \in X\setminus D$ be an unramified point, then we define
\begin{equation}\label{fiber}
\mathcal{L}_{r} \,:=\, \pi_r^*\mathcal{O}_X(x).
\end{equation}
Also for $y\,\in\,Y$ define 
\begin{equation}\label{fibert}
\widetilde{\mathcal{L}}_{r}\,:=\, \widetilde{\pi}_{r}^{\, \ast} \mathcal{O}_{Y}(y).
\end{equation}
If $p^{-1}(x)\,=\,\{y_1,\dots,y_{|\Gamma|}\}$,
then
\begin{equation}\label{capl}
\widetilde{L}_{r}\,=\, \widetilde{p}_r^{\,*}(\mathcal{L}_{r})
\,=\,
\widetilde{\pi}_r^{\,*}p^*\mathcal{O}_X(x)
\,=\,
\bigotimes_{i=1}^{|\Gamma|}\widetilde{\pi}_r^{\,*}\mathcal{O}_Y(y_i).
\end{equation}
Since $N^1(Y)\,=\,\R$, all points of $Y$ are numerically equivalent, and hence
\begin{equation}\label{eq:pullback-L}
c_1\bigl(\widetilde{L}_{r}\bigr)
\,=\,
|\Gamma|\, c_1\bigl( \widetilde{\mathcal{L}}_{r}\bigr).
\end{equation}
Similarly if $x^{\prime}\,\in\, D$ and let
\begin{equation}\label{fiberp}
\mathcal{L}^{'}_{r}\,:=\,\pi_r^*\mathcal{O}_X(x').\end{equation}

However, since $N^1(X)\,=\,\R$, so  $c_{1}(\mathcal{O}_{X}(x))$ and $c_{1}(\mathcal{O}_{X}(x'))$ are numerical equivalent classes in $ N^{1}(X)$. Therefore
\begin{equation}\label{equal}
c_{1}(\mathcal{L}_{r})\,=\, c_{1}(\mathcal{L}^{\prime}_{r})
\,\in\, N^1(\Gr(r,E_*)).
\end{equation}
Observe that for any unramified point $x \in X \setminus D$, the fiber $\pi_{r}^{-1}(x)\,=\,\mathrm{Gr}(r,\mathbb{C}^{k})$, and hence is smooth. In contrast, for $x' \in D$, the fiber $\pi_{r}^{-1}(x')$ may be a finite quotient of $\mathrm{Gr}(r,\mathbb{C}^{k})$. Thus, if $x\,\in\, X\setminus D$ and $y\,\in\, p^{-1}(x)$, then restricting \eqref{0} gives
\begin{equation}\label{eq:unramified-fiber}
\begin{tikzcd}
\widetilde{\pi}_r^{-1}(y)=\Gr(r,\C^k) \arrow[r, hook] \arrow[d, "\cong"']
& \Gr(r,\widetilde{E}) \arrow[d, "\widetilde{p}_r"] \\
\pi_r^{-1}(x) \arrow[r, hook]
& \Gr(r,E_*).
\end{tikzcd}
\end{equation}
In particular,
\begin{equation}\label{eq:taut-restrict-fiber}
\mathcal{O}_{\Gr(r,E_*)}(1)\big|_{\pi_r^{-1}(x)}
\,\cong\,
\mathcal{O}_{\Gr(r,\C^k)}(\lambda).
\end{equation}

\begin{convention}\label{con1}
We  adopt the following convention throughout the manuscript. Since $\grt$ is smooth, we will not distinguish between the groups
\begin{equation*}
N^{k}(\grt)\,=\,
N_{n-k}(\grt),
\end{equation*}
where $n\,=\,\dim(\grt)$. However, since $\grp$ need not be smooth (see Example~\ref{example}), we shall distinguish between the groups $N^{k}(\grp)$ and $N_{n-k}(\grp)$, where $n=\dim(\grp))$. More precisely, let $P\,\in\,N^{k}(\grp)$ be represented by a weighted degree $k$--Chern polynomial. Then the corresponding element of $N_{n-k}(\grp)$ will be denoted by $P\,\cap\,[\grp]$.
\end{convention}

\subsection{Nef Cone of the Parabolic Grassmann Bundle over a Curve.} Let $\e$ be a parabolic vector bundle over $X$ of rank at least two, say $k$ and parabolic divisor $D$.  Fix an integer $r\in [1,k-1]$. Let
\[
\Gr(r,E_{\ast}) \longrightarrow X
\]
be the parabolic Grassmann bundle defined in \eqref{pargrass}, and let
$\mathcal{O}_{\Gr(r,E_{\ast})}(1)$ be its tautological line bundle on $\textrm{Gr}(r, E_{\ast})$, defined in \eqref{tauto}.

We determine the numerical group $N^{1}(\Gr(r,E_{\ast}))$ and the nef cone $\Nef^{1}(\Gr(r,E_{\ast}))$.

\begin{proposition}\label{prop:N1}
The real vector space $N^1(\Gr(r,E_*))$ is two-dimensional and is generated by the classes
\[
N^{1}(\grpp)
\,=\,
\Big\langle
c_1\bigl(\mathcal{O}_{\Gr(r,E_*)}(1)\bigr),\,
c_1(\mathcal{L}_{r})
\Big\rangle,
\]
where $\mathcal{L}_{r}\,:=\, \pi_r^*\mathcal{O}_X(x)$ for some $x\,\in\, X\setminus D$.
\end{proposition}

\begin{proof}
Recall the notations in \S\ref{3.1}. We have the pullback map
\[
\widetilde{p}_r^{\,*}\,:\,
N^1(\grp)\,\longrightarrow\, N^1(\Gr(r,\widetilde{E})).
\] Using the projection formula and surjectivity of $ \widetilde{p}_r$, we see that $\widetilde{p}_r^{\,*} $  is injective.

It is well-known that $\textrm{dim}_{\R} N^1(\Gr(r,\widetilde{E}))\,=\,2,$ and that it is generated by the classes $c_1\bigl(\mathcal{O}_{\Gr(r,\widetilde{E})}(1)\bigr)
\,\text{and}\,
c_1\bigl(\widetilde{\pi}_r^{\,*}\mathcal{O}_Y(y)\bigr)$ for some $y\,\in\,Y$. Hence,
\[
\dim_{\R}N^1(\Gr(r,E_*))
\,\le\, 2.
\]
We now show that the classes $c_1\bigl(\mathcal{O}_{\Gr(r,E_*)}(1)\bigr)\,\,\text{and}\,\,
c_1(\mathcal{L}_{r})$ are linearly independent. Suppose that
\[
a\, c_1\bigl(\mathcal{O}_{\Gr(r,E_*)}(1)\bigr)
\,+\,
b\, c_1(\mathcal{L}_{r})
\,=\,0
\]
for some $a,b\,\in\, \R$. Pulling back via $\widetilde{p}_r^{\,*}$ and using \eqref{pullback} and \eqref{eq:pullback-L}, we obtain
\[
aN\, c_1\bigl(\mathcal{O}_{\Gr(r,\widetilde{E})}(1)\bigr)
\,+\,
b|\Gamma|\, c_1\bigl(\widetilde{\pi}_r^{\,*}\mathcal{O}_Y(y)\bigr)
\,=\,0
\]
in $N^1(\Gr(r,\widetilde{E}))$. Since these two classes are linearly independent, we conclude that
\[
a\,=\,b\,=\,0.
\]
Therefore,
\[
N^1(\Gr(r,E_*))
\,=\,
\R\cdot
c_1\bigl(\mathcal{O}_{\Gr(r,E_*)}(1)\bigr)
\oplus
\R\cdot
c_1(\mathcal{L}_{r}).
\]
This completes the proof.
\end{proof}
\begin{remark}\label{rem:Lx-Lxprime}
Let $x^{\prime}\,\in\,  D$, then by \eqref{equal}, we can write 
$$N^{1}(\grp)\,=\,\Big\langle c_1\bigl(\mathcal{O}_{\Gr(r,E_*)}(1)\bigr),\,\,    c_1(\mathcal{L}^{\prime}_{r}) \Big\rangle.$$
\end{remark}
We first prove the following lemma on the relative ampleness of $\mathcal{O}_{\Gr(r, E_{\ast})}(1)$.

\begin{lemma}\label{relamp}
The tautological line bundle $\mathcal{O}_{\Gr(r,E_*)}(1)$ on $\Gr(r,E_*)$ is relatively ample over $X$.
\end{lemma}

\begin{proof}
Let $x \in X$, and let $y \in p^{-1}(x)$ under the finite Galois morphism $p \,\colon\, Y \,\longrightarrow\, X.$ Then we have the following commutative diagram:
\begin{center}
\begin{tikzcd}
\widetilde{\pi}_r^{-1}(y)\, =\, \Gr(r,\C^k) \arrow[r, hook] \arrow[d] 
& \Gr(r,\widetilde{E}) \arrow[r, "\widetilde{\pi}_r"] \arrow[d, "\widetilde{p}_r"] 
& Y \arrow[d, "p"] \\
\pi_r^{-1}(x) \arrow[r, hook] 
& \Gr(r,E_*) \arrow[r, "\pi_r"] 
& X.
\end{tikzcd}
\end{center}
It is known that the line bundle $\mathcal{O}_{\Gr(r,\widetilde{E})}(1)$ is relatively ample over $Y$ (see \cite{BP}). Hence, $\mathcal{O}_{\Gr(r,\widetilde{E})}(\lambda)$ is also relatively ample over $Y$. By construction of $\mathcal{O}_{\textrm{Gr}(r, \e)}(1)$, we have
\[
\widetilde{p}_{r}^{\,*}\bigl(\mathcal{O}_{\Gr(r,E_*)}(1)\bigr)
\,=\,
\mathcal{O}_{\Gr(r,\widetilde{E})}(\lambda).
\]
Since $\widetilde{p}_{r}$ is a finite surjective morphism, it follows that 
$\mathcal{O}_{\Gr(r,E_*)}(1)$ is relatively ample over $X$.
\end{proof}
Consider the Harder--Narasimhan filtration of \(E_{\ast}\) defined in \eqref{HN}, and recall that \(m\) denotes its length. Let 
\[
t\,\in\, [1, m]
\]
be the unique largest integer such that
\begin{equation}\label{t}
\sum_{i=t}^{m}\rk(E^i/E^{i-1})\, \ge\, r.
\end{equation}
Equivalently, either $t\,=\, m$, or $t$ is the smallest integer such that
\begin{equation*}
\sum_{i\,=\,t+1}^{m}\,\textrm{rk}(E^{i}/E^{i-1})
\,=\,
\textrm{rk}(E/E^{t})
\,<\, r.
\end{equation*}
Define
\begin{equation}\label{thetap}
\theta_{E_*,r}
\,:=\,
\bigl(r-\rk(E/E^t)\bigr)\cdot 
\mathrm{par}\text{-}\mu\bigl((E^t/E^{t-1})_*\bigr)
\,+\,
\mathrm{par}\text{-}\textrm{deg}(E/E^t)_* .
\end{equation}
Now consider the Harder--Narasimhan filtration of $\widetilde{E}$ defined in \eqref{hno}. We define
\begin{equation}\label{thetat}
\theta_{\widetilde{E},r}
\,:=\,
(r\,-\,\textrm{rk}(\w{E}/\w{E}^{t}))
\cdot
\mu(\w{E}^{t}/\w{E}^{t-1})
\,+\,
\textrm{deg}(\w{E}/\w{E}^{t}),
\end{equation}
(see \cite{BP} for further details). Then one can see that
\begin{equation}\label{eq:theta-compare}
\theta_{\widetilde{E},r}
\,=\,
|\Gamma|\,\theta_{E_*,r}.
\end{equation}
Note that $\theta_{\e,r}$ in \eqref{thetap} need not be an integer, it is clear that $\theta_{\e,r}$ is well defined as an element in $N^{1}(\grp)$ because $\theta_{\e,r}\,\in\, \mathbb{Q}$.

\begin{lemma}\label{positivity}
The tautological line bundle $\mathcal{O}_{\Gr(r,E_*)}(1)$ satisfies the following:
\begin{enumerate}
\item[(i)] If $\theta_{E_*,r}\,>\,0$, then $\mathcal{O}_{\Gr(r,E_*)}(1)$ is ample.
\item[(ii)] If $\theta_{E_*,r}\,=\,0$, then $\mathcal{O}_{\Gr(r,E_*)}(1)$ is nef but not ample.
\item[(iii)] If $\theta_{E_*,r}\,<\,0$, then $\mathcal{O}_{\Gr(r,E_*)}(1)$ is not nef.
\end{enumerate}
\end{lemma}

\begin{proof}
By \eqref{eq:theta-compare}, $\theta_{\widetilde{E},r}$ and $\theta_{E_*,r}$ have the same sign. By \cite[Theorem~3.4]{BP}, the line bundle $\mathcal{O}_{\Gr(r,\widetilde{E})}(1)$ on $\Gr(r,\widetilde{E})$ is ample, nef but not ample, or not nef according as $\theta_{\widetilde{E},r}$ is positive, zero, or negative, respectively. Since $\lambda\,>\,0$, the line bundle $\mathcal{O}_{\Gr(r,\widetilde{E})}(\lambda)$ has the same positivity properties. Using \eqref{pullback}, we have
\[
\widetilde{p}_r^{\,*}\bigl(\mathcal{O}_{\Gr(r,E_*)}(1)\bigr)
\,=\,
\mathcal{O}_{\Gr(r,\widetilde{E})}(\lambda).
\]
Since $\widetilde{p}_r$ is proper (resp., finite) and surjective, $\mathcal{O}_{\Gr(r,E_*)}(1)$ is ample (resp., nef). This proves $(i)$, $(ii)$, and $(iii)$.
\end{proof}
Next, we compute the numerical group $N^{1}(\Gr(r,E_{\ast}))$ and the nef cone $\Nef^{1}(\Gr(r,E_{\ast}))$.

\begin{theorem}\label{thm:nef-cone}
The boundary of the nef cone of $\Gr(r,E_*)$ in
$N^{1}(\emph{Gr}(r,\e))$ is given by
$c_1\bigl(\mathcal{O}_{\Gr(r,E_*)}(1)\bigr)
-\lambda\theta_{E_*,r}\,c_{1}(\mathcal{L}_{r})$
and $c_1(\mathcal{L}_{r})$,
where $x\, \in\, X\setminus D$ and $\lambda$ is defined in \eqref{n}.
\end{theorem}

\begin{proof}
Fix an unramified point $x\,\in\, X\,\setminus\, D$, and let $y\,\in\, p^{-1}(x)$. For the usual Grassmann bundle $\Gr(r,\widetilde{E})$, \cite[Proposition 4.1]{BP} shows that
\[
\Nef(\Gr(r,\widetilde{E}))
\,=\,
\left\langle
c_1\bigl(\mathcal{O}_{\Gr(r,\widetilde{E})}(1)\bigr)-\theta_{\widetilde{E},r}\,c_1\bigl(\widetilde{\mathcal{L}}_{r}\bigr), \,c_1\bigl(\widetilde{\mathcal{L}}_{r}\bigr)
\right\rangle,
\]
thus, both generators lie on its boundary. Since, we know  
$$\widetilde{p}_r^{\,*}\bigl(c_1(\mathcal{L}_{r}))\,=\, c_{1}(\w{p}_{r}^{\,\ast}(\mathcal{L}_{r})).$$ By \eqref{capl}, \eqref{eq:pullback-L}, and \eqref{pullback},
we obtain the following equalities in
$N^{1}(\textrm{Gr}(r, \e))$.
\[
 c_{1}(\w{p}_{r}^{\,\ast}(\mathcal{L}_{r}))
\,=\,
|\Gamma|\, c_1\bigl(\widetilde{\pi}_r^{\,*}\mathcal{O}_Y(y)\bigr)
\]
and
\[
\widetilde{p}_r^{\,*}\,\left(
c_1\bigl(\mathcal{O}_{\Gr(r,E_*)}(1)\bigr)-\lambda\theta_{E_*,r}\,c_1(\mathcal{L}_{r})
\right)
\,=\,
\lambda\, c_1\bigl(\mathcal{O}_{\Gr(r,\widetilde{E})}(1)\bigr)
-\lambda
|\Gamma|\theta_{E_*,r}\, c_1\bigl(\widetilde{\mathcal{L}}_{r}\bigr),
\]
for some $y\,\in\,Y$.
Using \eqref{eq:theta-compare}
\begin{equation*}
\hspace{6.5cm} =\,
\lambda\left(
c_1\bigl(\mathcal{O}_{\Gr(r,\widetilde{E})}(1)\bigr)
\,-\,
\theta_{\widetilde{E},r}\, c_1\bigl(\widetilde{\mathcal{L}}_{r}\bigr)
\right).
\end{equation*}
Therefore, both the classes
\begin{equation}\label{nefclasses}
c_1(\mathcal{L}_{r})
\quad\text{and}\quad
c_1\bigl(\mathcal{O}_{\Gr(r,E_*)}(1)\bigr)-\lambda\theta_{E_*,r}\,c_1(\mathcal{L}_{r})
\end{equation}
pull back to positive multiples of nef classes on $\Gr(r,\widetilde{E})$. Since $\widetilde{p}_r$ is proper and surjective, both classes in \eqref{nefclasses} are nef on $\Gr(r,E_*)$. Since $c_1\bigl(\mathcal{O}_{\Gr(r,\widetilde{E})}(1)\bigr)
\,-\,
\theta_{\widetilde{E},r}c_1\bigl(\widetilde{\mathcal{L}}_{r}\bigr)$  and  $c_1\bigl(\widetilde{\mathcal{L}}_{r}\bigr) $ lies on the boundary of $ \textrm{Nef}(\Gr(r, \widetilde{E}))$, they
are not ample.  As $\widetilde{p}_{r}$ is a finite surjective map, this implies that both $c_{1}(\mathcal{O}_{\Gr(r,E_*)}(1))-\theta_{\e, r} c_{1}(\mathcal{L}_{r})$ and $c_{1}(\mathcal{L}_{r})$ are also not ample. Thus, both the classes lie on the boundary of the nef cone. Since they are linearly independent and nef but not ample, they generate the two extremal rays of $\Nef(\Gr(r,E_*))$.
\end{proof}

\subsection{The pseudoeffective cone of the parabolic Grassmann bundle }
In this subsection, we will study the  pseudoeffective cone of Parabolic  Grassmann bundle. Let $\e$ be a parabolic bundle of rank $k$ and recall its Harder-Narasimhan filtration as in $\eqref{HN}$. Fix an integer
and define $v\,\in\,[0, m-1]$ to be the unique smallest integer such that
\[
\sum_{i=1}^{v+1} \rk(E^i/E^{i-1})\,>\, r.
\]
Note that if $\e$ is semistable, then $v\,=\, 0$.
Define
\begin{equation}\label{eq:zeta-par}
\zeta_{E_*,r}
\,:=\,
(r - \rk(E^{v}))\cdot \mathrm{par}\text{-}\mu\bigl((E^{v+1}/E^{v})_*\bigr)
\,+\,
\mathrm{par}\text{-}\mathrm{deg}(E^{v}_*).
\end{equation}
Let
\[
\pi_r\, \colon\, \Gr(r,E_*)\, \longrightarrow\, X
\]
be the projection map and the dimension of $\Gr(r,E_*)\,=\,:n_{r}$. Let $\mathcal{O}_{\textrm{Gr}(r, \e)}(1)\, \longrightarrow\, \textrm{Gr}(r, \e)$ 
be the tautological line bundle and  $ \mathcal{L}_{r}\, =\, \pi_{r}^{\ast}(\mathcal{O}_{X}(x))$ with $ x \,\in\, X\setminus D$ and  $\mathcal{L}^{'}_{r}\, =\, \pi_{r}^{\ast}(\mathcal{O}_{X}(x^{\prime}))$ with $ x\,\in\, D$.
Let
\[
\widetilde{\pi}_r \,\colon\, \Gr(r,\widetilde{E})\,\longrightarrow\, Y
\]
be the projection map for the corresponding orbifold bundle $\w{E}$ and let the quotient map be
\[
\widetilde{p}_r \,\colon\, \Gr(r,\widetilde{E}) \longrightarrow \Gr(r,E_*).
\]
Consider the commutative diagram
\begin{equation}\label{eq:pseudo-diagram}
\begin{tikzcd}
\Gr(r,\widetilde{E}) \arrow[r, "\widetilde{\pi}_r"] \arrow[d, "\widetilde{p}_r"']
& Y \arrow[d, "p"] \\
\Gr(r,E_*) \arrow[r, "\pi_r"']
& X.
\end{tikzcd}
\end{equation}
Consider $\w{\mathcal{L}}_r$ and
$\w{{L}}_{r}$ in \eqref{fibert} and \eqref{capl}.
Define
\begin{equation*}
\zeta_{\widetilde{E},r} \,:=\, \deg(\widetilde{E}^{v})\, =\, |\Gamma|\, \zeta_{E_*,r}.
\end{equation*}
In \cite[Theorem 4.1]{BHP}, it is shown that the pseudoeffective cone 
$\Gr(r,\widetilde{E})$ admits the following description
\begin{equation}\label{pseudt}
\overline{\textrm{Eff}}^{1}(\Gr(r,\widetilde{E}))\,=\,\overline{\textrm{Eff}}_{n_{r}-1}(\textrm{Gr}(r, \widetilde{E}))\,=\,
\Big\langle c_1\bigl(\mathcal{O}_{\Gr(s,\widetilde{E})}(1)\bigr)
-
\zeta_{\widetilde{E},r}\, c_1(\widetilde{\mathcal{L}}_{r}),\, \,\,c_1(\widetilde{\mathcal{L}}_{r})\Big\rangle.
\end{equation}
\begin{theorem}\label{thm:pseudo}
The boundary of the pseudoeffective cone $\overline{\textnormal{Eff}}_{n_{r}-1}(\Gr(r,E_*))$ is generated by
\[
c_1(\mathcal{L}_{r}) \,\cap\, [\Gr(r,E_*)]
\quad\text{and}\quad
\Big(c_1\bigl(\mathcal{O}_{\Gr(r,E_*)}(1)\bigr)
-
\zeta_{E_*,r}\, c_1(\mathcal{L}_{r})\Big)\, \cap\, [\Gr(r,E_*)].
\] 

Furthermore, the pushforward map $\widetilde p_*\,:\,
\overline{\Eff}_{n_r-1}(\Gr(r,\widetilde E))\,
\longrightarrow\,
\overline{\Eff}_{n_r-1}(\Gr(r,E_*))$ is an isomorphism.

\end{theorem}

\begin{proof}
By \eqref{pseudt}, it is clear that the boundary of the pseudoeffective cone  $\overline{\textnormal{Eff}}_{n_{r}-1}(\Gr(r,\widetilde{E}))$ is generated by
\[
c_1(\widetilde{\mathcal{L}}_{r})
\quad\text{and}\quad
c_1\bigl(\mathcal{O}_{\Gr(r,\widetilde{E})}(1)\bigr)
\,-\
\zeta_{\widetilde{E},r}\, c_1(\widetilde{\mathcal{L}}_{r}).
\]
Using \cite[Proposition 3.5]{BBM1} , it follows that
\[
\widetilde{p}_{r,\mathrm{sp}}^{\,*}
\Bigl(
c_1\bigl(\mathcal{O}_{\Gr(r,E_*)}(1)\bigr)
\,\cap\,
[\Gr(r,E_*)]
\Bigr)
\,=\,
\lambda\,
c_1\bigl(\mathcal{O}_{\Gr(r,\widetilde{E})}(1)\bigr),
\]
and
\[
\widetilde{p}_{r,\mathrm{sp}}^{\,*}
\Bigl(
c_1(\mathcal{L}_{r})
\,\cap\,
[\Gr(r,E_*)]
\Bigr)
\,=\,
c_1(\widetilde{L}_{r}).
\]
Now applying the identity  $\widetilde{p}_{r,*}
\circ
\widetilde{p}_{r,\mathrm{sp}}^{\,*}
\,=\,
|\Gamma|\,\mathrm{Id},$ in \eqref{comp} we obtain
\begin{equation}\label{pse1}
\widetilde{p}_{r,*}
\Bigl(
c_1\bigl(\mathcal{O}_{\Gr(r,\widetilde{E})}(1)\bigr)
\Bigr)
\,=\,
\frac{|\Gamma|}{\lambda}
\Bigl(
c_1\bigl(\mathcal{O}_{\Gr(r,E_*)}(1)\bigr)
\,\cap\,
[\Gr(r,E_*)]
\Bigr),
\end{equation}
and
\begin{equation}\label{pse2}
\widetilde{p}_{r,*}
\bigl(c_1(\widetilde{L}_{r})\bigr)
\,=\,
|\Gamma|
\Bigl(
c_1(\mathcal{L}_{r})
\,\cap\,
[\Gr(r,E_*)]
\Bigr).
\end{equation}
Also, note that in $N_{n_{r}-1}(\Gr(r, \widetilde{E}))$  we have $ c_{1}(\widetilde{L}_{r})\, =\, |\Gamma| c_{1}(\widetilde{\mathcal{L}}_{r})$ and 
$$\left(
c_1(\mathcal{O}_{\Gr(r,\widetilde{E})}(1)  - \zeta_{\widetilde{E}, r} \, c_1(\widetilde{\mathcal{L}}_{r}) 
\right)\,=\, \left(
c_1(\mathcal{O}_{\Gr(r,\widetilde{E})}(1)  - \zeta_{E_{\ast}, r} \, c_1(\widetilde{{L}}_{r}) 
\right) $$
Now, using \eqref{pse1}, \eqref{pse2} and Remark \ref{pdecent} we have
\[
\widetilde{p}_{r_ {*}}\!\left(
c_1(\mathcal{O}_{\Gr(r,\widetilde{E})}(1)  - \zeta_{\widetilde{E}, r} \, c_1(\widetilde{\mathcal{L}}_{r}) 
\right)
\,=\,
\frac{|\Gamma|}{\lambda}
\Big(
(c_1(\mathcal{O}_{\textrm{Gr}(E_{\ast}, r)}(1)) 
 - \zeta_{E_{\ast}, r}\, c_1(\mathcal{L}_{r}))
\, \cap \,[\Gr(r,E_*)]
\Big)\]
and
\[\widetilde{p}_{r_{*}}\bigl(c_1(\w{\mathcal{{L}}}_{r})\bigr)\,=\,
c_1(\mathcal{L}_{r})\, \cap\, [\Gr(r,E_*)].
\]
Thus, by \cite[Corollary 3.8]{FL1}, the classes $\Big(
(c_1(\mathcal{O}_{\textrm{Gr}(E_{\ast}, r)}(1)) 
 - \zeta_{{E}_{\ast}, r}\, c_1(\mathcal{L}_{r}))
\, \cap \,[\Gr(r,E_*)]
\Big)$ and $\Big(c_1(\mathcal{L}_{r})\, \cap\, [\Gr(r,E_*)]\Big)$ are not numerically trivial. Consider the weight $n_{r}-1$--Chern polynomial $\Big(c_{1}(\mathcal{O}_{\textrm{Gr}(r, \e)}(1))^{n_{r}-2} \cdot c_{1}(\mathcal{L}_{r})\Big)$,
we have 
\begin{equation*}
\Big(c_{1}(\mathcal{O}_{\textrm{Gr}(r, \e)}(1))^{n_{r}-2} \cdot c_{1}(\mathcal{L}_{r})\Big)  \,\cap\, \Big(c_{1}(\mathcal{L}_{r})\, \cap\, [\Gr(r, E_{\ast})] \Big)\,=\, 0,
\end{equation*}
and on the other hand we have
\[
\Big(c_{1}(\mathcal{O}_{\textrm{Gr}(r, \e)}(1)^{n_{r}-2} \cdot c_{1}(\mathcal{L}_{r})\Big) \,\cap\, 
\Big(\left(
c_1(\mathcal{O}_{\Gr(r,E_*)}(1)) 
-
\zeta_{E_*,r}\, c_1(\mathcal{L}_{r})
\right)\, \cap\, [\Gr(r,E_*)] \Big)
\]
\begin{equation*}
\,=\, \Big(c_{1}(\mathcal{O}_{\textrm{Gr}(r, \e)}(1))^{n_{r}-1} \cdot c_{1}(\mathcal{L}_{r})\Big)\,\cap\, [\Gr(r,E_*)]\,=\,{\textrm{deg}(\textrm{Gr}(r, \mathbb{C}^{k}))\lambda^{n_{r}-1}\,\neq\,0}.    
\end{equation*}
This shows that these two classes are not numerically equivalent and linearly independent. Thus, by \cite[Corollary 3.22]{FL2} these two classes generate the boundary of the pseudoeffective cone. Since $\widetilde p_{\ast}$ maps the generators of
$\overline{\mathrm{Eff}}_{n_r-1}(\Gr(r,\widetilde E))$ to the generators of
$\overline{\mathrm{Eff}}_{n_r-1}(\Gr(r,E_{\ast}))$,
it follows that $\widetilde p_*\,:\,
\overline{\Eff}_{n_r-1}(\Gr(r,\widetilde E))\,
\longrightarrow\,
\overline{\Eff}_{n_r-1}(\Gr(r,E_*))$ is an isomorphism.
\end{proof}

\begin{corollary}
The dual pseudoeffective cone of $\Gr(r,E_*)$ is given by 
\[
\overline{\mathrm{Eff}}^{1}(\Gr(r,E_*))
\,=\,
\left\langle
c_1(\mathcal{L}_{r}),\;
c_1(\mathcal{O}_{\Gr(r,E_*)}(1))
-
\zeta_{E_*,r}\, c_1(\mathcal{L}_{r})
\right\rangle.
\]
\end{corollary}
\begin{proof}
This follows using the definition of $\overline{\mathrm{Eff}}^{1}(\Gr(r,E_*))$  and the above theorem.
\end{proof}
\begin{corollary}
The nef cone of $\Gr(r,E_{\ast})$ coincides with $\overline{\mathrm{Eff}}^{1}(\Gr(r,E_*))$ if and only if the
vector bundle $E_{\ast}$ is parabolic semistable.
\end{corollary}
\begin{proof}
We note that $\theta_{E_\ast,r}$ coincides with $\zeta_{E_\ast,r}$ if and only if $m\,=\,1$, meaning that $E_\ast$ is semistable.
\end{proof}

\subsection{The Mori cone of the parabolic Grassmann bundle}
Let $X$ be a smooth complex projective curve, and let $\e$ be a parabolic vector bundle on $X$ which is parabolic unstable. Let
\begin{equation*}
E^m\, \supset\, E^{m-1}\, \supset\, \cdots\, \supset\, E^0\,=\,0
\end{equation*}
be the Harder--Narasimhan filtration of $E_*$ defined in \eqref{HN}. Since $\e$ is assumed to be parabolic unstable, we have $m\,\ge\,2$. Fix an integer $l$ satisfying $1\,<\, l\, \le\, m$ and
\[
q \,:=\, {\rk}(E/E^{l-1}).
\]
Let
\[
\pi_q\,:\,\Gr(q,E_*) \,\longrightarrow\, X,
\] 
be the corresponding parabolic Grassmann bundle and set
\[
n_q\,:=\,\dim(\Gr(q,E_*)).
\]
Let
\begin{equation}\label{thetapm}
\theta
\,:=\,
\mathrm{par}\textrm{-deg}(E/E^{l-1})_*.
\end{equation}
Replacing $r$ by $q$, we obtain the line bundles $\mathcal{L}_q$ and $\mathcal{L}'_{q}$ on $\Gr(q,E_*)$ as defined in \eqref{fiber} and \eqref{fiberp}, respectively. By \eqref{equal}, the line bundles $\mathcal{L}_q$ and $\mathcal{L}_q^{\prime}$ determine the same numerical class. Consider the following numerical classes in $N_{n_q-1}(\Gr(q,E_*))$:
\begin{equation}\label{morip}
M_q
\,:=\,
\Bigl(
c_1\bigl(\mathcal{O}_{\Gr(q,E_*)}(1)\bigr)
\,-\,
\theta\lambda\, c_1(\mathcal{L}_q)
\Bigr)
\,\cap\,
[\Gr(q,E_*)],
\end{equation}
and
\[
H_q
\,:=\,
c_1(\mathcal{L}_q)\,\cap\, [\Gr(q,E_*)].
\]
Let
\[
\widetilde{\pi}_q\,:\,\Gr(q,\widetilde{E}) \,\longrightarrow\, Y
\]
be the Grassmann bundle associated to $\widetilde{E}$ over $Y$, and let $\widetilde{p}_q\,:\,\Gr(q,\widetilde{E}) \,\longrightarrow\,\Gr(q,E_*)$ be the quotient morphism. Then we have the following commutative diagram
\begin{equation}\label{eq:mori-diagram}
\begin{tikzcd}
\Gr(q,\widetilde{E}) \arrow[r, "\widetilde{\pi}_q"] \arrow[d, "\widetilde{p}_q"']
& Y \arrow[d, "p"] \\
\Gr(q,E_*) \arrow[r, "\pi_q"']
& X.
\end{tikzcd}
\end{equation}
Let $\widetilde{\mathcal{L}}_{q}$ and $\widetilde{L}_{q}$
be the line bundles on $\Gr(q,\widetilde{E})$
defined in \eqref{fibert} and \eqref{capl}, respectively,
after replacing $r$ by $q$. Consider the Harder--Narasimhan filtration of $\widetilde{E}$ given in \eqref{hno}, and recall that $m$ denotes its length. By assumption, $m\,\ge\,2$, and hence $\widetilde{E}$ is unstable. Analogously to \eqref{thetapm}, define
\[
\widetilde{\theta}
\,:=\,
\deg(\widetilde{E}/\widetilde{E}^{\,l-1}).
\]
Clearly,
\[
\widetilde{\theta}
\,=\,
|\Gamma|\,\theta.
\]
Since $\Gr(q,\widetilde{E})$ is smooth, we may identify
\[
N^1(\Gr(q,\widetilde{E}))
\,=\,
N_{n_q-1}(\Gr(q,\widetilde{E})).
\]
Define
\begin{equation}\label{mh}
\widetilde{M}_q
\,:=\,
c_1\bigl(\mathcal{O}_{\Gr(q,\widetilde{E})}(1)\bigr)
-
\widetilde{\theta}\,
c_1(\widetilde{\mathcal{L}}_q),
\qquad
\widetilde{H}_q
\,:=\,
c_1(\widetilde{\mathcal{L}}_q)
\,=\,
[\widetilde{\pi}_{q}^{-1}(y)],
\end{equation}
where $[\widetilde{\pi}_{q}^{-1}(y)]$ denotes the numerical class of the fiber $\widetilde{\pi}_{q}^{-1}(y)$ in $N_{n_q-1}(\Gr(q,\widetilde{E}))$ for some $y\,\in\,Y$. Let $\widetilde{\Gamma}_{\ell}$ denote the numerical class of a line in a fiber $\widetilde{\pi}_{q}^{-1}(y)$, $y\,\in\,Y$. Let $\widetilde{\Gamma}_{\widetilde{s}}$ denote the numerical class of the image of the section 
\[
\widetilde{s}\,:\,Y\,\longrightarrow\,\Gr(q,\widetilde{E})
\]
induced by the rank $q$ quotient $\widetilde{E}/\widetilde{E}^{\,l-1}$.
Recall from \cite[Section~3]{BHNN} that the Mori cone of
$\Gr(q,\widetilde{E})$ is given by
\[
\overline{\mathrm{Eff}}_{1}(\Gr(q,\widetilde{E}))
\,=\,
\left\langle
\widetilde{\Gamma}_{\ell},\,\,
\widetilde{\Gamma}_{\widetilde{s}}
\right\rangle .
\]
We now give another description of the Mori cone of $\grtq$, which will be useful in computing the Mori cone of $\grp$. Define
\begin{equation}\label{morit}
\alpha_q
\,:=\,
\widetilde{M}_q^{\,n_q-1}
\in
N_1(\Gr(q,\widetilde{E})),
\qquad
\beta_q
\,:=\,
\widetilde{M}_q^{\,n_q-2}\cdot \widetilde{H}_q
\in
N_1(\Gr(q,\widetilde{E})),
\end{equation}
where $\widetilde{M}_{q}$ and $\widetilde{H}_{q}$ are defined in \eqref{mh}.

\begin{lemma}\label{lem:orbifold-mori}
The Mori cone of $\Gr(q,\widetilde{E})$ is given by
\[
\overline{\mathrm{Eff}}_1(\Gr(q,\widetilde{E}))
\,=\,
\Big\langle \alpha_q,\,\,\beta_q\Big\rangle,
\]
where $\alpha_q$ and $\beta_q$ are defined in \eqref{morit}.
\end{lemma}

\begin{proof}
Clearly, $\alpha_q$ and $\beta_q$ are pseudoeffective, since they are obtained as intersections of nef classes. We first show that $\alpha_q$ lies on the ray generated by
$\widetilde{\Gamma}_{\widetilde{s}}$. Since $\widetilde{M}_q$ generates the nef cone of $\Gr(q,\widetilde{E})$ (see, \cite[Proposition 4.1]{BP}), it is nef but not ample. Thus,
\[
\widetilde{M}_q\cdot \alpha_q \,=\, 0.
\]
On the other hand,
\[
\widetilde{H}_q\cdot \alpha_q
\,=\,
\widetilde{H}_q\cdot \widetilde{M}_q^{\,n_q-1}
\,=\,
\widetilde{H}_q\cdot
c_1\bigl(\mathcal{O}_{\Gr(q,\widetilde{E})}(1)\bigr)^{n_q-1}
\,>\,0,
\]
because $\mathcal{O}_{\Gr(q,\widetilde{E})}(1)$ is relatively ample over $Y$. Let 
\[
\alpha_q
\,=\,
a\,\widetilde{\Gamma}_{\widetilde{s}}
+
b\,\widetilde{\Gamma}_{\ell},
\qquad a,b\,\ge\, 0.
\]
Intersecting with $\widetilde{M}_q$, we obtain
\[
0
\,=\,
\widetilde{M}_q\cdot \alpha_q
\,=\,
a\bigl(\widetilde{M}_q\cdot \widetilde{\Gamma}_{\widetilde{s}}\bigr)
+
b\bigl(\widetilde{M}_q\cdot \widetilde{\Gamma}_{\ell}\bigr).
\]
Since
\[
\widetilde{M}_q\cdot \widetilde{\Gamma}_{\widetilde{s}}\,=\,0
\quad\text{and}\quad
\widetilde{M}_q\cdot \widetilde{\Gamma}_{\ell}\,>\,0,
\]
it follows that $b\,=\,0$. Moreover, since
\[
\widetilde{H}_q\cdot \alpha_q\,>\,0,
\]
we must have $a\,>\,0$. Thus, $\alpha_q$ lies on the ray generated by $\widetilde{\Gamma}_{\widetilde{s}}$. Since $\widetilde{H}_q^2\,=\,0$, we have
\[
\widetilde{H}_q\cdot \beta_q \,=\, 0.
\]
Write
\[
\beta_q
\,=\,
a'\,\widetilde{\Gamma}_{\widetilde{s}}
\,+\,
b'\,\widetilde{\Gamma}_{\ell},
\quad a',b'\,\ge\, 0.
\]
Intersecting with $\widetilde{H}_q$, we obtain
\[
0
\,=\,
\widetilde{H}_q\cdot \beta_q
\,=\,
a'\bigl(\widetilde{H}_q\cdot \widetilde{\Gamma}_{\widetilde{s}}\bigr)
\,+\,
b'\bigl(\widetilde{H}_q\cdot \widetilde{\Gamma}_{\ell}\bigr).
\]
Since
\[
\widetilde{H}_q\cdot \widetilde{\Gamma}_{\widetilde{s}}\,>\,0
\quad\text{and}\quad
\widetilde{H}_q\cdot \widetilde{\Gamma}_{\ell}\,=\,0,
\]
it follows that $a'\,=\,0$. Also,
\[
\widetilde{M}_q\cdot \beta_q\,>\,0,
\]
and hence $b'\,>\,0$. Therefore, $\beta_q$ lies on the ray generated by $\widetilde{\Gamma}_{\ell}$. Consequently,
$$\overline{\mathrm{Eff}}_1(\Gr(q,\widetilde{E}))
\,=\,
\Big\langle \alpha_q,\,\,\beta_q\Big\rangle.$$ This completes the proof.
\end{proof}

\begin{theorem}\label{thm:mori-parabolic}
The Mori cone of $\Gr(q,E_*)$ is given by
\[
\overline{\mathrm{Eff}}_1(\Gr(q,E_*))
\,=\,
\left\langle
M_q^{\,n_q-1},\;
M_q^{\,n_q-2}\cdot H_q
\right\rangle,
\]
where the classes $M_q$ and $H_q$ are defined in \eqref{morip}.
\end{theorem}

\begin{proof}
It is known that the special pullback map
\[
\widetilde{p}_{q,\mathrm{sp}}^{\,*}
\,:\,
CH_*(\Gr(q,E_*))
\,\longrightarrow\,
CH_*(\Gr(q,\widetilde{E}))
\]
defined in \eqref{special} is a ring homomorphism, it follows from \cite[Proposition 3.5]{BBM1} that
\begin{equation}\label{push1}
\widetilde{p}_{q,\mathrm{sp}}^{\,*}\bigl(M_q^{\,n_q-1}\bigr)
\,=\,
\lambda^{n_q-1}\,\widetilde{M}_q^{\prime\,n_q-1},
\end{equation}
where
\[
\widetilde{M}_q^{\prime}
\,=\,
c_1\bigl(\mathcal{O}_{\Gr(q,\widetilde{E})}(1)\bigr)
-
\widetilde{\theta}\,
c_1(\widetilde{L}_q).
\]
Similarly,
\[
\widetilde{p}_{q,\mathrm{sp}}^{\,*}
\bigl(M_q^{\,n_q-2}\cdot H_q\bigr)
\,=\,
{\lambda^{n_q-2}}\,
\widetilde{M}_q^{\prime\,n_q-2}
\cdot
\widetilde{H}^{\prime}_{q},
\]
where $\widetilde{H}^{\prime}_{q}\,=\, c_{1}(\widetilde{L}_{q})$. After replacing $r$ by $q$ in \eqref{eq:pullback-L}, we obtain
\begin{equation}\label{push2}
\widetilde{p}_{q,\mathrm{sp}}^{\,*}
\bigl(M_q^{\,n_q-2}\cdot H_q\bigr)
\,=\,
\lambda^{n_q-2}|\Gamma|\,
\widetilde{M}_q^{\,n_q-2}\cdot \widetilde{H}_q.
\end{equation}
By Lemma~\ref{lem:orbifold-mori}, the curve classes $\alpha_{q}\,=\,\widetilde{M}_q^{\,n_q-1}$\,\text{and}\,
$\beta_{q}\,=\,\widetilde{M}_q^{\,n_q-2}\cdot \widetilde{H}_q$ generate the boundary of the Mori cone of $\Gr(q,\widetilde{E})$. Now using the identity $\widetilde{p}_{q,*}\circ \widetilde{p}_{q,\mathrm{sp}}^{\,*}
\,=\,
|\Gamma|\,\mathrm{Id}$ in \eqref{comp}, together with \eqref{push1} and \eqref{push2}, we obtain
\[
\widetilde{p}_{q,*}
\bigl(\widetilde{M}_q^{\,n_q-1}\bigr)
\,=\,
\frac{|\Gamma|}{\lambda^{n_{q}-1}}\,M_q^{\,n_q-1},
\]
and
\[
\widetilde{p}_{q,*}
\bigl(\widetilde{M}_q^{\,n_q-2}\cdot \widetilde{H}_q\bigr)
\,=\,
\frac{1}{\lambda^{n_q-2}}\,
\bigl(M_q^{\,n_q-2}\cdot H_q\bigr).
\]
Since we have $  \overline{\mathrm{Eff}}_1(\Gr(q,\widetilde{E}))\,\xlongrightarrow{\widetilde{p}_{q,\ast}}\,\overline{\mathrm{Eff}}_1(\Gr(q,E_*))$, thus $M_q^{\,n_q-1}$ and
$M_q^{\,n_q-2}\cdot H_q$ are effective classes on $\Gr(q,E_*)$. Moreover, by \cite[Corollary~3.2]{FL1}, these classes are not numerically trivial. It remains to show that they are linearly independent. \text{red}{change these lines} Since $c_1(\widetilde{L}_q)^2\,=\,0$. Consider the class $c_{1}(\mathcal{L}_{q})\in N^{1}(\Gr(q,\e))$. We obtain
\[
c_1(\mathcal{L}_q)
\,\cap\,
\bigl(M_q^{\,n_q-2}\cdot H_q\bigr)
\,=\,
M_q^{\,n_q-2}\cdot H_q^2.
\]
Thus,
\[
\deg\!\left(
c_1(\mathcal{L}_q)
\,\cap\,
\bigl(M_q^{\,n_q-2}\cdot H_q\bigr)
\right)
\,=\,
0.
\]
On the other hand,
\[
\deg\!\left(
c_1(\mathcal{L}_q)\,\cap\, M_q^{\,n_q-1}
\right)
=\,1. 
\]
Therefore, $M_q^{\,n_q-1}$ and $M_q^{\,n_q-2}\cdot H_q$ are linearly independent. Now by \cite[Corollary 3.22]{FL2}, it follows that $M_q^{\,n_q-1}$ and
$M_q^{\,n_q-2}\cdot H_q$ generates the Mori cone of $\Gr(q, \e)$.
\end{proof}

\section{Fiber product of parabolic Grassmann bundles}
Let $X$ be a smooth projective curve, and let $D\subset X$ be a reduced effective divisor. 
Let $E_{1*}$ and $E_{2*}$ be parabolic vector bundles of ranks $k_{1}$ and $k_{2}$, respectively, on $X$, endowed with parabolic structures along $D$. Assume that $E_{1*}$ and $E_{2*}$ satisfy \cite[Assumption~3.2]{Bi1}. Let $\w{E}_{1}$ and $\w{E}_{2}$
be the corresponding orbifold bundle on $Y$.
Fix integers
\[
1\,\le\, r_1\,\le\, k_1-1\quad \textrm{and}
\quad
1\,\le\, r_2\,\le\, k_2-1.
\]
Let
\[
\pi_{r_1}\,\colon\, \Gr(r_1,E_1{_*})\longrightarrow X,\quad \textrm{and}\quad
\pi_{r_2}\,\colon\, \Gr(r_2,E_2{_*})\,\longrightarrow\, X
\]
be the corresponding parabolic Grassmann bundles. On the orbifold side, let
\[
\widetilde{\pi}_{r_1}\colon \Gr(r_1,\widetilde E_1)\longrightarrow Y,
\qquad
\widetilde{\pi}_{r_2}\colon \Gr(r_2,\widetilde E_2)\longrightarrow Y
\]
be the $\Gamma$--Grassmann bundles.
Consider the diagonal action of $\Gamma$ on 
\begin{equation*}
\w{S}\,=\,\Gr(r_1,\widetilde E_1)\,\times_Y\, \Gr(r_2,\widetilde E_2).
\end{equation*}
Now take the $\Gamma$--quotient of $\w{S}$, we have:
\begin{equation*}
S\,:=\,\w{S}/\Gamma\,=\,  \Big(\Gr(r_1,\widetilde E_1)\,\times_Y\, \Gr(r_2,\widetilde E_2)\Big)/\Gamma\,=\,\Gr(r_1,\widetilde E_1)/\Gamma \,\times_{X}\, \Gr(r_2,\widetilde E_2)/\Gamma 
\end{equation*}
\begin{equation*}
\hspace{8cm}=\,\Gr(r_1, E_1{_*})\,\times_X\, \Gr(r_2,E_2{_*}).    
\end{equation*}
Let $\widetilde p\,\colon\, \widetilde S\,\longrightarrow\,S \,\,\,\textrm{denotes the quotient morphism}.$ We also denote the natural projection morphisms by
\[
\widetilde{\varepsilon}_1 \,\colon\, \widetilde{S}\,\longrightarrow\,\Gr(r_1,\widetilde{E}_1),
\qquad
\widetilde{\varepsilon}_2 \,\colon\, \widetilde{S}\,\longrightarrow\,\Gr(r_2,\widetilde{E}_2),
\]
and
\[
\varepsilon_1 \,\colon\, S\,\longrightarrow\, \Gr(r_1, E_{1_*}),
\qquad
\varepsilon_2 \,\colon\, S\,\longrightarrow\, \Gr(r_2, E_{2_*}).
\]

Consider the following commutative diagram
\begin{equation}
\hspace{-1.5cm}
\begin{tikzcd}
Y \arrow[d, "p"]    
&  \wee \arrow[l, "\widetilde{\pi}_{r_2}"] \arrow[d, "\widetilde{p}_{r_2}"]    
& \widetilde{S} \,:=\, \wS
\arrow[l, "\widetilde{\varepsilon}_{2}"]
\arrow[d, "\widetilde{p}"]
\arrow[r,"\widetilde{\varepsilon}_{1}"]   
& \we
\arrow[r, "\widetilde{\pi}_{r_{1}}"]
\arrow[d, "\widetilde{p}_{r_1}"]    
& Y\arrow[d, "p"]\\
X                        
&  \paee  \arrow[l, "\pi_{r_2}"]              
& S \,:=\,\pS
\arrow[r, "\varepsilon_{1}"]
\arrow[l, "\varepsilon_{2}"]           
& \pae
\arrow[r, "\pi_{r_1}"]           
& X.
\end{tikzcd}
\end{equation}

\subsection{Tautological line bundles}
Let $N(E_{1\ast})$ and $N(E_{2\ast})$ denote the least common multiples of the orders of the isotropy subgroups arising from the actions of $\textrm{GL}(k_{1},\mathbb C)$ and $\textrm{GL}(k_{2},\mathbb C)$ on the ramified principal bundles
$E_{1,\textrm{GL}(k_{1},\mathbb C)}$ and
$E_{2, \textrm{GL}(k_{2},\mathbb C)}$
corresponding to the parabolic bundles
$E_{1\ast}$ and $E_{2\ast}$, respectively. For simplicity of notation, we set
\[
\lambda_1\,:=\,N(E_{1_*}),
\qquad
\lambda_2\,:=\,N(E_{2_*}).
\]
By construction of the tautological line bundles, we have
\begin{equation}\label{eq:taut-pullback-1}
\widetilde{p}^{\,*}\Bigl(\varepsilon_1^*\mathcal{O}_{\Gr(r_1,E_{1_*})}(1)\Bigr)
\,=\,
\widetilde{\varepsilon}_1^{\,*}
\mathcal{O}_{\Gr(r_1,\widetilde{E}_1)}(\lambda_1),
\end{equation}
and
\begin{equation}\label{eq:taut-pullback-2}
\widetilde{p}^{\,*}\Bigl(\varepsilon_2^*\mathcal{O}_{\Gr(r_2,E_{2_*})}(1)\Bigr)
\,=\,
\widetilde{\varepsilon}_2^{\,*}
\mathcal{O}_{\Gr(r_2,\widetilde{E}_2)}(\lambda_2).
\end{equation}

\subsection{Divisors induced from the base curve}

Let \(x\in X\setminus D\) be an unramified point. Define line bundles on \(S\) by
\[
\mathcal L
\,:=\,
(\pi_{r_1}\circ \varepsilon_1)^*\mathcal O_X(x)\,=\,
(\pi_{r_2}\circ \varepsilon_2)^*\mathcal O_X(x).
\]

Similarly, for a point \(x'\,\in\, D\), define
\[
\mathcal L'
\,:=\,
(\pi_{r_1}\circ \varepsilon_1)^*\mathcal O_X(x')
\,=\,
(\pi_{r_2}\circ \varepsilon_2)^*\mathcal O_X(x').
\]

We now describe their pullbacks to \(\widetilde S\).
Let $x\,\in\, X\setminus D$
\[
p^{-1}(x)\,=\,\{y_1,\ldots,y_{|\Gamma|}\},
\qquad y_i\,\in\, Y.
\]
Then
\[
\widetilde L
\,:=\,
\widetilde p^{\,*}\mathcal L
\,=\,
(\widetilde\pi_{r_1}\circ \widetilde\varepsilon_1)^*
p^*\mathcal O_X(x)
\,=\,
\bigotimes_{i=1}^{|\Gamma|}
(\widetilde\pi_{r_1}\circ \widetilde\varepsilon_1)^*
\mathcal O_Y(y_i)\,=\,(\widetilde\pi_{r_2}\circ \widetilde\varepsilon_2)^*
p^*\mathcal O_X(x).
\]

Now let \(x'\,\in\, D\), and write
\begin{equation*}
p^{-1}(x')
\,=\,
\{y'_1,\ldots,y'_t\},
\qquad y'_i\,\in\, Y.
\end{equation*}
Assume that \(y'_i\) occurs with multiplicity \(e_i\) for each
\(i=1,\ldots,t\), so that
\[
e_1\,+\,\cdots\,+\,e_t\,=\,|\Gamma|.
\]

Then
\[
\widetilde L'
\,:=\,
\widetilde p^{\,*}\mathcal L'
\,=\,
\widetilde\varepsilon_1^{\,*}
\Bigl(
\widetilde\pi_{r_1}^{\,*}\mathcal O_Y(y'_1)^{e_1}
\otimes\cdots\otimes
\widetilde\pi_{r_1}^{\,*}\mathcal O_Y(y'_t)^{e_t}
\Bigr).
\]
Since \(X\) and \(Y\) are smooth projective curves, any two points on \(X\) are numerically equivalent, and the same holds for \(Y\). Consequently,
\[
c_1\!\bigl(\mathcal O_X(x)\bigr)
\,=\, 
c_1\!\bigl(\mathcal O_X(x')\bigr)
\qquad\text{in}\,\,N^1(X).
\]
Thus,
\begin{equation*}
c_{1}(\mathcal{L})\,=\, c_{1}(\mathcal{L}^{'}),\qquad c_{1}(\widetilde{L})\,=\,|\Gamma| c_{1}(\widetilde{\mathcal{L}}), \quad\textrm{and} \quad c_{1}(\widetilde{L}^{\prime})\,=\,|\Gamma| c_{1}(\widetilde{\mathcal{L}})
\end{equation*}
in $N^{1}(S)$ and $N^{1}( \widetilde{S}),$ respectively.

Therefore, for the purpose of studying divisor and cycle cones on \(S\), it
is sufficient to work with the numerical class associated with an
unramified point \(x\in X\setminus D\).

\begin{remark}
Although the divisor classes arising from ramified and unramified points
are numerically equivalent, the geometry of the corresponding fibers is
different. If \(x\in X\setminus D\) is unramified, then the fiber of the
natural projection
\[
S\, \xlongrightarrow{\pi_{r}\,\circ\,\varepsilon_{1}}\, X
\]
is isomorphic to
\[
\Gr(r_1,\C^{k_1}) \,\times\, \Gr(r_2,\C^{k_2}). 
\]
On the other hand, if \(x\in D\), the isotropy subgroup associated with
the covering \(p\,:\,Y\,\longrightarrow\, X\) acts nontrivially on this product, and the
fiber is obtained as the corresponding quotient. Thus the distinction
between ramified and unramified points remains geometrically significant,
even though it disappears at the level of numerical equivalence classes.
\end{remark}
\subsection{The Nef and Pseudoeffective Cones of the Fiber Product of Parabolic Grassmann Bundles} Let $E_{1*}$ and $E_{2*}$ be parabolic vector bundles over $X$. For $j\,=\,1,2$, let
\begin{equation}\label{hnpj}
E_j
\,=\,
E_j^{m_j}
\supsetneq
E_j^{m_j-1}
\supsetneq
\cdots
\supsetneq
E_j^0
\,=\,
0
\end{equation}
be the Harder--Narasimhan filtration of $E_{j*}$. Thus, for each
$i\,=\,1,\ldots,m_j$, the parabolic quotient
\[
Q^i_{j*}
\,:=\,
\left(E_j^i/E_j^{i-1}\right)_*
\]
is parabolic semistable, and the corresponding parabolic slopes satisfy
\[
\mu_{j1}\,>\,\mu_{j2}\,>\,\cdots\,>\,\mu_{jm_j}.
\]

By the correspondence between parabolic bundles and orbifold bundles, the Harder--Narasimhan filtrations correspond to each other. Hence the corresponding orbifold bundles $\widetilde E_1$ and $\widetilde E_2$ admit the Harder--Narasimhan filtrations corresponding to \eqref{hnpj}. 

For $j\,=\,1,2$, let
\begin{equation}\label{hnoj}
\widetilde E_j
\,=\,
\widetilde E_j^{m_j}
\supsetneq
\widetilde E_j^{m_j-1}
\supsetneq
\cdots
\supsetneq
\widetilde E_j^0
\,=\,
0
\end{equation}
be the Harder--Narasimhan filtration of $\widetilde E_j$. Thus, for each
$i\,=\,1,\ldots,m_j$, the quotient
\[
\widetilde Q_j^i
\,:=\,
\widetilde E_j^i/\widetilde E_j^{i-1}
\]
is semistable, and the corresponding slopes satisfy
\[
\widetilde\mu_{j1}
\,>\,
\widetilde\mu_{j2}
\,>\,
\cdots
\,>\,
\widetilde\mu_{jm_j}.
\]


Let the Harder--Narasimhan filtration of $E_{1*}$ and $E_{2*}$ be as in
\eqref{hnpj} 
let $m_1$ and $m_2$ denote its length. Fix integers
\[
r_1\,\in\, [1,\,k_{1}-1]\quad ,\quad r_2\,\in\, [1,\,k_{2}-1].
\]
Let $t_1$ and $t_{2}$ be the largest integer such that
\[
\sum_{i=t_1}^{m_1}
\rk(E_1^i/E_1^{i-1})
\,\ge\, r_1 \quad , \quad \sum_{i=t_2}^{m_2}
\rk(E_2^i/E_2^{i-1})
\,\ge\, r_2.
\]
Define
\[
\theta_{E_{1*},r_1}
\,:=\,
\bigl(r_1-\rk(E_1/E_1^{t_1})_*\bigr)
\cdot
\mathrm{par}\text{-}\mu\!\left((E_1^{t_1}/E_1^{t_1-1})_*\right)
+
\mathrm{par}\text{-}\deg(E_1/E_1^{t_1})_*
\]
and 
\begin{equation*}
\theta_{E_{2*},r_2}
\,:=\,
\bigl(r_2-\rk(E_2/E_2^{t_2})_*\bigr)
\cdot
\mathrm{par}\text{-}\mu\!\left((E_2^{t_2}/E_2^{t_{2}-1})_*\right)
+
\mathrm{par}\text{-}\deg(E_2/E_2^{t_2})_*.
\end{equation*}
Analogously, for orbifold bundles $\widetilde{E}_{1}$ and $ \widetilde{E}_{2}$ we can define $ \theta_{\widetilde{E}_{1}, r_{1}}$, and $ \theta_{\widetilde{E}_{2}, r_{2}}$. Also, we note that 
\begin{equation}\label{th}
\theta_{\widetilde{E}_{1}, r_{1}}\,=\, |\Gamma|\theta_{{E_{1}}_{\ast},r_{1}}   \quad  \textrm{and} \quad \theta_{\widetilde{E}_{2}, r_{2}}\, =\, |\Gamma| \theta_{{E_{2}}_{\ast},r_{2}}.
\end{equation}

\begin{theorem}\label{nef cone prod}
The nef cone of $S \,=\, \Gr(r_1,E_{1*})\, \times_X\, \Gr(r_2,E_{2*})$ is given by
\[
\emph{Nef}(S)
\,=\,
\Big\langle
c_1(\varepsilon_{1}^{\ast}\Oponee)-\lambda_1\theta_{E_{1*},r_1}\,c_1(\mathcal{L}),\,\,\,
c_1(\varepsilon_{2}^{\ast}\Optwoo)-\lambda_2\theta_{E_{2*},r_2}\,c_1(\mathcal{L}),\,\,\,
c_1(\mathcal{L})
\Big\rangle.
\]
Moreover, each of the classes $c_1(\varepsilon_{1}^{\ast}\Oponee)-\lambda_1\theta_{E_{1*},r_1}\,c_1(\mathcal{L}),\,\,$
$c_1(\varepsilon_{2}^{\ast}\Optwoo)-\lambda_2\theta_{E_{2*},r_2}\,c_1(\mathcal{L}),\,\,$ and $c_1(\mathcal{L})$ lie on the boundary of $\emph{Nef}(S)$.
\end{theorem}

\begin{proof}
By \cite[Proposition~5.5]{MR}, the nef cone of $\widetilde{S}$ is given by
\begin{equation*}
\textrm{Nef}(\widetilde{S})\,=\,\Big \langle c_1\bigl(\widetilde{\varepsilon}_{1}^{\,\ast}\mathcal{O}_{\Gr(r_1,\widetilde{E}_1)}(1)\bigr)
-
\theta_{\widetilde{E}_1,r_1}\,\,\,
c_1(\widetilde{\mathcal{L}}),\,\,c_1\bigl(\widetilde{\varepsilon}_{2}^{\,\ast}\mathcal{O}_{\Gr(r_2,\widetilde{E}_2)}(1)\bigr)
-
\theta_{\widetilde{E}_2,r_2}\,
c_1(\widetilde{\mathcal{L}}),\,\,\,  c_1(\widetilde{\mathcal{L}})\Big\rangle.
\end{equation*}Consider the pullback map
\[
\widetilde{p}^{\,*}
\,:\,
N^1(S)
\,\longrightarrow\,
N^1(\widetilde{S}).
\]

Using \cite[Proposition 3.5]{BBM1}, together with
\eqref{eq:taut-pullback-1}  \eqref{eq:taut-pullback-2}, and \eqref{th} we obtain 
$$\widetilde{p}^{\,\ast}(c_{1}(\mathcal{L}))\,=\,c_{1}(L)\,=\,|\Gamma|c_{1}(\widetilde{\mathcal{L}}),$$
$$\widetilde{p}^{\,*}\Bigl(
c_1({\varepsilon}_{1}^{\,\ast}\Oponee)
-
\lambda_1\theta_{E_{1*},r_1}\,
c_1(\mathcal{L})
\Bigr)
\,=\, \lambda_1\Bigl(
c_1\bigl(\widetilde{\varepsilon}_{1}^{\,\ast}\mathcal{O}_{\Gr(r_1,\widetilde{E}_1)}(1)\bigr)
-
\theta_{\widetilde{E}_1,r_1}\,
c_1(\widetilde{\mathcal{L}})
\Bigr),$$
and
$$\widetilde{p}^{\,*}\Bigl(
c_1(\varepsilon_{2}^{\ast}\Optwoo)
-
\lambda_2\theta_{E_{2*},r_2}\,
c_1(\mathcal{L})
\Bigr)\,=\,\lambda_2\Bigl(
c_1\bigl(\widetilde{\varepsilon}_{2}^{\,\ast}\mathcal{O}_{\Gr(r_2,\widetilde{E}_2)}(1)\bigr)
-
\theta_{\widetilde{E}_2,r_2}\,
c_1(\widetilde{\mathcal{L}})
\Bigr).$$
Since the above classes lie on the boundary of
\(\textrm{Nef}(\widetilde{S})\), they are nef but not ample.
As \(\widetilde{p}\) is a proper surjective morphism, nefness descends along \(\widetilde{p}\). Hence the classes $c_1(\varepsilon_{1}^{\ast}\Oponee)-\lambda_1\theta_{E_{1*},r_1}\,c_1(\mathcal{L}),$ $c_1(\varepsilon_{2}^{\ast}\Optwoo)-\lambda_2\theta_{E_{2*},r_2}\,c_1(\mathcal{L}),$ and $c_1(\mathcal{L})$ are nef but not ample on \(S\). Therefore, they lie on the boundary of \(\textrm{Nef}(S)\), proving the theorem.
\end{proof}
Next, we compute the pseudoeffective cone of $S$. We first set up the necessary notation. For $j\,=\,1,2$, let
\[
v_j\,\in\, [0,m_j-1]
\]
be the unique smallest integer such that
\[
\sum_{i=1}^{v_j+1}
\rk(E_j^i/E_j^{i-1})_*
\,>\,
r_j.
\]
Define
\begin{equation}\label{ztj}
\zeta_{E_{j*},r_j}
\,:=\,
\bigl(r_j-\rk(E_j^{v_j})\bigr)\cdot
\mathrm{par}\text{-}\mu\!\left((E_j^{v_j+1}/E_j^{v_j})_*\right)
+
\mathrm{par}\text{-}\deg(E_{j*}^{\,v_j}),
\qquad j\,=\,1,2.
\end{equation}
Analogously, for $j\,=\,1,2$, we define
$\zeta_{\widetilde E_j,r_j}$ for the orbifold bundle
$\widetilde E_j$. By the correspondence between parabolic and orbifold bundles, we have
\[
\zeta_{\widetilde E_j,r_j}
\,=\,
|\Gamma|\,
\zeta_{E_{j*},r_j},
\qquad
j\,=\,1,2.
\]
Let
\[
n_{r_1r_2}
\,:=\,
\dim(\widetilde S).
\]
Since the quotient morphism
\[
\widetilde p\,:\,\widetilde S\,\longrightarrow\, S
\]
is finite, we have 
\[
\dim(\widetilde{S})\,=\,\dim( S)\,=\,n_{r_1r_2}.
\]

\begin{theorem}
The boundary of $\overline{\emph{Eff}}_{n_{r_{1}r_{2}}-1}(S)$ is generated by the classes
$$\left(c_{1}(\varepsilon_{1}^{\ast}\mathcal{O}_{\Gr(r_{1}, {E_{1}}_{\ast})}(1))\,-\, \lambda_{1}\zeta_{{E_{1}}_{\ast}, r_{1}}\,c_{1}(\mathcal{L})\right)\,\cap\,[S], \quad\left(c_{1}(\varepsilon_{2}^{\ast}\mathcal{O}_{\Gr(r_{2}, {E_{2}}_{\ast})}(1))\,-\, \zeta_{{E_{2}}_{\ast}, r_{2}}\,c_{1}(\mathcal{L})\right)\,\cap\,[S]$$ and $c_{1}(\mathcal{L})\,\cap\,[S]$. Here  $\zeta_{{E_{1}}_{\ast}, r_{1}}$ and $ \zeta_{{E_{2}}_{\ast}, r_{2}}$ are defined in \eqref{ztj}. Furthermore, the pushforward map $\widetilde p_*\,:\,
\overline{\Eff}_{n_{r_{1}r_{2}-1}}(\widetilde{S})\,
\longrightarrow\,
\overline{\Eff}_{n_{r_{1}r_{2}-1}}(S)$ is an isomorphism.
\end{theorem}

\begin{proof}
By \cite[Theorem 4.8]{MR}, the pseudoeffective cone of $\widetilde{S}$ has the following description
\[
\overline{\textrm{Eff}}_{n_{r_{1}r_{2}}-1}(\widetilde{S})
\,=\,
\Bigg\langle
\left(
c_{1}(\widetilde{\varepsilon}_{1}^{\,\ast}\mathcal{O}_{\Gr(r_{1}, \widetilde{E}_{1})}(1))
-\zeta_{\widetilde{E}_{1},r_{1}}\, c_{1}(\widetilde{\mathcal{L}})
\right),
\left(
c_{1}(\widetilde{\varepsilon}_{2}^{\,\ast}\mathcal{O}_{\Gr(r_{2}, \widetilde{E}_{2})}(1))
-\zeta_{\widetilde{E}_{2},r_{2}}\, c_{1}(\widetilde{\mathcal{L}})
\right),
c_{1}(\widetilde{\mathcal L})
\Bigg\rangle.
\]
Moreover, in $ \textrm{N}_{n_{r_{1}r_{2}}-1}(\widetilde{S})$ we have
$c_1\bigl(\widetilde{L}\bigr)
\,=\,
|\Gamma|\, c_1\bigl( \widetilde{\mathcal{L}}\bigr)$. From the equation 
\begin{equation*}
\zeta_{\widetilde E_j,r_j}
\,=\,
|\Gamma|\,
\zeta_{E_{j*},r_j},
\,\,\,
j\,=\,1,2
\end{equation*}
we see that in $\textrm{N}_{n_{r_{1}r_{2}}-1}(\widetilde{S})$ we have
$$\left(
c_{1}(\widetilde{\varepsilon}_{1}^{\,\ast}\mathcal{O}_{\Gr(r_{1}, \widetilde{E}_{1})}(1))
-\zeta_{\widetilde{E}_{1},r_{1}}\, c_{1}(\widetilde{\mathcal{L}})
\right)\,=\,\left(c_{1}(\widetilde{\varepsilon}_{1}^{\,\ast}{ \mathcal{O}_{\Gr(r_{1},\widetilde{E}_{1})}(1)             })- {\zeta_{E_{1}}}_{\ast,r_{1}}\, c_{1}(\widetilde{{L}})\right),$$ 
and
$$
\left(
c_{1}(\widetilde{\varepsilon}_{2}^{\,\ast}\mathcal{O}_{\Gr(r_{2}, \widetilde{E}_{2})}(1))
-\zeta_{\widetilde{E}_{2},r_{2}}\, c_{1}(\widetilde{\mathcal{L}})
\right)\,=\,\left(c_{1}(\widetilde{\varepsilon}_{2}^{\,\ast}{\mathcal{O}_{\Gr(r_{2},\widetilde{E}_{2})}(1)})- {\zeta_{E_{2}}}_{\ast,r_{2}}\, c_{1}(\widetilde{{L}})\right).$$ Using the identity $\widetilde{p}_{\ast}\circ \widetilde{p}^{\,\ast}_{\textrm{sp}}\,=\,|\Gamma|\textrm{Id}$ and \cite[Proposition 3.5]{BBM1}, it follows that under the map $\widetilde{p}_{\ast}\,:\,N_{n_{r_{1}r_{2}}-1}(\widetilde{S})\,\longrightarrow\,N_{n_{r_{1}r_{2}}-1}(S)$, we have 
\begin{equation*}
\begin{aligned}
{\p}_{\ast}\left(c_{1}(\widetilde{\varepsilon}_{1}^{\,\ast}{ \mathcal{O}_{\Gr(r_{1},\widetilde{E}_{1})}(1)             })- {\zeta_{E_{1}}}_{\ast,r_{1}}\, c_{1}(\widetilde{{L}})\right)
&\,=\, \frac{|\Gamma|}{\lambda_{1}}\Big(\left(c_{1}({\varepsilon}_{1}^{\,\ast}\mathcal{O}_{\Gr(r_{1}, {E_{1}}_{\ast})}(1))\,-\, \lambda_{1}\zeta_{{E_{1}}_{\ast}, r_{1}}\, c_{1}(\mathcal{L})\right)\,\cap\,[S]\Big), \\
{\p}_{\ast}\left(c_{1}(\widetilde{\varepsilon}_{2}^{\,\ast}{\mathcal{O}_{\Gr(r_{2},\widetilde{E}_{2})}(1)})- {\zeta_{E_{2}}}_{\ast,r_{2}}\, c_{1}(\widetilde{{L}})\right)
&\,=\, \frac{|\Gamma|}{\lambda_{2}}\Big((c_{1}(\varepsilon_{2}^{\ast}\mathcal{O}_{\Gr(r_{2},{E_{2}}_{\ast})}(1)) \,-\,\lambda_{2} {\zeta_{E_{2}}}_{\ast,r_{2}}\, c_{1}(\mathcal{L}))\,\cap\, [S]\Big), \\
{\p}_{\ast} \left(c_{1}(\widetilde{L})\right)
&\,=\, c_{1}(\mathcal{L}).
\end{aligned}
\end{equation*}
  These classes are pseudoeffective, as pushforword of pseudoeffective cycle is pseudoeffective . Now, using the same argument as in the proof of Theorem~\ref{thm:pseudo}, we show that these pseudoeffective cycles on $S$ are linearly independent and generate the pseudoeffective cone.
\end{proof}

\begin{corollary}
The dual pseudoeffective cone of $S$ is 
\[
\overline{\emph{Eff}}^{1}(S)
\,=\,
\Big\langle\Big(
c_{1}(\varepsilon_{1}^{\ast}\mathcal{O}_{\Gr(r_{1},{E_{1}}_{\ast})}(1))
- \zeta_{{E_{1}}_{\ast}, r_{1}}
\, c_{1}(\mathcal{L})
\Big),
\Big(
c_{1}(\varepsilon_{2}^{\ast}\mathcal{O}_{\Gr(r_{2},{E_{2}}_{\ast})}(1))
- \zeta_{{E_{2}}_{\ast}, r_{2}}
\, c_{1}(\mathcal{L}),\, c_{1}(\mathcal{L})
\Big)
\Big\rangle .
\]
\end{corollary}

\begin{proof}
This follows using the definition of $\overline{\mathrm{Eff}}^{1}(S)$  and the above theorem. 
\end{proof}

\subsection{Mori cone of the fiber product of Parabolic Grassmann bundles}
In
this subsection, we describe the Mori cone of the fiber product of parabolic grassmann bundle over smooth projective curve $X$. Let ${E_{1}}_{\ast}$ and ${E_{2}}_{\ast}$ be  parabolic bundles on $X$ which are parabolic unstable (i.e., not parabolic semistable).

The Harder--Narasimhan filtrations of $E_{1*}$ and $E_{2*}$ are given in
\eqref{hnpj}, with lengths $m_1$ and $m_2$, respectively. Since $E_{1*}$ and $E_{2*}$ are parabolic unstable, we have
\[
m_1,m_2\,\ge\, 2.
\]

For $j\,=\,1,2$, fix an integer
\[
1<l_j\,\le\, m_j,
\]
and define
\[
q_j
\,:=\,
\rk(E_j/E_j^{\,l_j-1}),
\qquad
\theta_j
\,:=\,
\operatorname{par-}\deg(E_j/E_j^{\,l_j-1})_*.
\]
Let $\widetilde{E}_{1}$ and $\widetilde{E}_{2}$ be the orbifold bundles corresponding to
$E_{1*}$ and $E_{2*}$, respectively. Define
\[
\widetilde\theta_j
\,:=\,
\deg(\widetilde E_j/\widetilde E_j^{\,l_j-1}),
\qquad
j\,=\,1,2.
\]
As in the beginning of this section, consider the fiber product of the $\Gamma$--Grassmann bundles
\[
\widetilde S
\,:=\,
\Gr(q_1,\widetilde E_1)
\,\times_Y\,
\Gr(q_2,\widetilde E_2).
\]
By the correspondence between parabolic and orbifold bundles, the diagonal action of $\Gamma$ on $\widetilde S$ induces a quotient morphism
\[
\widetilde p\,:\,\widetilde S\,\longrightarrow\, S,
\]
where
\[
S
\,:=\,
\Gr(q_1,E_{1*})
\,\times_X\,
\Gr(q_2,E_{2*})
\]
is the fiber product of the corresponding parabolic Grassmann bundles over \(X\). Denote its dimension by $n_{q_{1}q_{2}}$.

We follow the notation for line bundles introduced in Subsections~4.1 and~4.2, replacing the indices $r_{1}$ and $r_{2}$ by $q_{1}$ and $q_{2}$, respectively.

Now we compute the Mori cone of $\widetilde{S}$. Let $\mathcal{O}_{\Gr(q_{1},\widetilde{E}_{1})}(1)$ and $\mathcal{O}_{\Gr(q_{2},\widetilde{E}_{2})}(1)$ be the tautological line bundles over $\Gr(q_{1},\widetilde{E}_{1})$ and
$\Gr(q_{2},\widetilde{E}_{2}),$ respectively. Define the numerical classes
\begin{equation}\label{m}
\widetilde{M}_{1}\,=\,c_{1}\bigl( \widetilde{\varepsilon}_{1}^{\, \ast}\mathcal{O}_{\Gr(q_{1},\widetilde{E}_{1})}(1)\bigr)
-
\widetilde{\theta}_{1}\,
c_{1}\bigl(\widetilde{\mathcal{L}}\bigr),
\end{equation}

\begin{equation*}
\widetilde{M}_{2}\,=\,c_{1}\bigl(\widetilde{\varepsilon}_{2}^{\, \ast}\mathcal{O}_{\Gr(q_{2},\widetilde{E}_{2})}(1)\bigr)
-
\widetilde{\theta}_{2}\,
c_{1}\bigl(\widetilde{\mathcal{L}}\bigr),
\end{equation*}
and
\begin{equation}\label{h}
\widetilde{H}\,=\,c_{1}(\widetilde{\mathcal{L}})
\end{equation}
in $ N^{1}(\widetilde{S})$.
Suppose that the ranks of vector bundle $\widetilde{E}_{1}$ and $ \widetilde{E}_{2}$ are $k_{1}$ and $ k_{2}$, respectively. Then 
\begin{equation*}
\textrm{dim}\,\, \widetilde{S}\,=\, (q_{1}(k_{1}\,-\,q_{1})\,+\,q_{2}(k_{2}\,-\,q_{2}))\,+\,1\,=\,n_{q_{1}q_{2}}
\end{equation*}
Set
\[
a\,=\, q_{1}(k_{1}\,-\,q_{1})
\quad \textrm{and} \quad
b\,=\,q_{2}(k_{2}\,-\,q_{2}).
\]

Then $a$ and $b$ are the dimensions of fibers of the maps $ \widetilde{\pi}_{q_{1}}$ and $ \widetilde{\pi}_{q_{2}}$, respectively.

\begin{proposition}\label{morif}

The Mori cone of $\widetilde{S}$ is generated by the classes 

$$\overline{\emph{Eff}}_{1}(\widetilde{S})\,=\,\Big\langle \widetilde{M}_{1}^{\,a-1} \widetilde{M}_{2}^{\,b} \widetilde{H}, \,\, \widetilde{M}_{1}^{\,a} \widetilde{M}_{2}^{\,b-1} \widetilde{H}, \,\, \widetilde{M}_{1}^{a} \widetilde{M}_{2}^{b}\Big\rangle.$$
    
\end{proposition}
\begin{proof}
We first show that the classes $\widetilde{M}_{1}, \widetilde{M}_{2}$ and    $\widetilde{H}$ are nef. Since $c_{1}(\mathcal{O}_{Y}(y))$, for some $y\,\in\,Y$ is nef, we have
\[
\widetilde{H}
\,=\,
c_{1}\!\left(
(\widetilde{\pi}_{q_{1}} \circ \widetilde{\varepsilon}_{1})^{*}
\bigl(\mathcal{O}_{Y}(y)\bigr)
\right) \,=\, \big[
(\widetilde{\pi}_{q_{1}} \circ \widetilde{\varepsilon}_{1})^{-1}(y)
\big].
\]
Hence $\widetilde{H}$ is nef. By \cite[Lemma $2.2$]{BHNN},
\[
c_{1}\!(
\mathcal{O}_{\Gr(q_{1},\widetilde{E}_{1})}(1)
)
-
\widetilde{\theta}_{1}\,
\widetilde{\pi}_{q_{1}}^{*}
c_{1}\!(\mathcal{O}_{Y}(y))
\]
and
\[
c_{1}\!
(\mathcal{O}_{\Gr(q_{2},\widetilde{E}_{2})}(1)
)
-
\widetilde{\theta}_{2}\,
\widetilde{\pi}_{q_{2}}^{*}
c_{1}\!\left(\mathcal{O}_{Y}(y)\right)
\]
are nef classes on $\Gr(q_{1},\widetilde{E}_{1})$ and
$\Gr(q_{2},\widetilde{E}_{2})$, respectively. Therefore, their pullbacks to $\widetilde{S}$ are also nef. Hence $\widetilde{M}_{1},$
$\widetilde{M}_{2}$, and $\widetilde{H}$, are nef classes on $\widetilde{S}$. Using Lemma 2.2 of \cite{BHNN}, together with \cite[Theorem, 5.5]{MR}  we conclude that $\widetilde{M}_{1},$ $\widetilde{M}_{2},$ and $\widetilde{H}$ generate the nef cone of $\widetilde{S}$.  Observe that
\[
\widetilde{M}_{1}
\left(
\widetilde{M}_{1}^{\,a-1}
\widetilde{M}_{2}^{\,b}
\widetilde{H}
\right)
\,=\,
\widetilde{M}_{1}^{\,a}
\widetilde{M}_{2}^{\,b}
\widetilde{H}.
\]
For any $y\,\in\,Y$, we have
\[
(\widetilde{\pi}_{q_{1}}\circ \widetilde{\varepsilon}_{1})^{-1}(y)
\,=\,
\Gr(q_{1},\mathbb{C}^{k_{1}})
\times_{\mathbb{C}}
\Gr(q_{2},\mathbb{C}^{k_{2}}).
\]
Moreover,
\[
\widetilde{M}_{1}\Big|
_{\Gr(q_{1},\mathbb{C}^{k_{1}})
\times_{\mathbb{C}}
\Gr(q_{2},\mathbb{C}^{k_{2}})}
\,=\,
c_{1}\!\left(
\widetilde{\varepsilon}_{1}^{\,*}
\mathcal{O}_{\Gr(q_{1},\mathbb{C}^{k_{1}})}(1)
\right),
\]
and
\[
\widetilde{M}_{2}\Big|
_{\Gr(q_{1},\mathbb{C}^{k_{1}})
\times_{\mathbb{C}}
\Gr(q_{2},\mathbb{C}^{k_{2}})}
\,=\,
c_{1}\!\left(
\widetilde{\varepsilon}_{2}^{\,*}
\mathcal{O}_{\Gr(q_{2},\mathbb{C}^{k_{2}})}(1)
\right).
\]
Using Lemma \ref{relamp}, we have 
\[
\widetilde{M}_{1}^{\,a}
\widetilde{M}_{2}^{\,b}
\widetilde{H}
\,=\,
\Big(c_{1}\!\left(
\widetilde{\varepsilon}_{1}^{\,*}
\mathcal{O}_{\Gr(q_{1},\mathbb{C}^{k_{1}})}(1)
\right)^{a}\Big)
\cdot
 \Big( c_{1}\!\left(
\widetilde{\varepsilon}_{2}^{\,*}
\mathcal{O}_{\Gr(q_{2},\mathbb{C}^{k_{2}})}(1)
\right)^{b} \Big)
\,>\,0.
\]
Clearly,
\[
\widetilde{M}_{1}^{\,a+1}
\widetilde{M}_{2}^{\,b}
\widetilde{H}
\,=\,
0
\,=\,
\widetilde{M}_{1}^{\,a}
\widetilde{M}_{2}^{\,b+1}
\widetilde{H}.
\]
Using the above computations, we  conclude that the curve classes
\[
\widetilde{M}_{1}^{\,a-1}\widetilde{M}_{2}^{\,b}\widetilde{H},
\qquad
\widetilde{M}_{1}^{\,a}\widetilde{M}_{2}^{\,b-1}\widetilde{H},
\quad\textrm{and}\quad
\widetilde{M}_{1}^{\,a}\widetilde{M}_{2}^{\,b}
\]
generate  the cone $\overline{\textrm{Eff}}_{1}(\widetilde{S}).$
\end{proof}
Next, we discuss the closed cone of curves of $S$. To simplify notation, define
\[
M_{1}
\,=\,
\Big(c_{1}\!\left(
\varepsilon_{1}^{*}
\bigl(\mathcal{O}_{\Gr(q_{1},E_{1}{_{\ast}})}(1)\bigr)
\right)
-
\theta_{1}\lambda_{1}\,
c_{1}(\mathcal{L})\Big)\,\cap\, [S],
\]
\[
M_{2}
\,=\,
\Big(c_{1}\!\left(
\varepsilon_{2}^{*}
\bigl(\mathcal{O}_{\Gr(q_{2},E_{2}{_{\ast}})}(1)\bigr)
\right)
-
\theta_{2}\lambda_{2}\,
c_{1}(\mathcal{L})\Big)\,\cap\, [S],
\]
and
\[
H
\,=\,
c_{1}(\mathcal{L})\,\cap\, [S].
\]
\begin{theorem}\label{morifp}
The Mori cone of $S$ is generated by
\begin{equation*}
\overline{\emph{Eff}}_{1}({S})\,=\,\Big\langle M_{1}^{\,a-1}M_{2}^{\,b}H,\,\,
M_{1}^{\,a}M_{2}^{\,b-1}H,\,\,
M_{1}^{\,a}M_{2}^{\,b}\Big\rangle.
\end{equation*}
\end{theorem}

\begin{proof}
As the special pullback map
\[
\widetilde{p}_{\textrm{sp}}^{\,*}
\,:\,
CH_{\ast}(S)
\,\longrightarrow\,
CH_{\ast}(\widetilde{S})
\]
is a ring homomorphism, it follows from \cite[Proposition 3.5]{BBM1}, that
\[
\widetilde{p}_{\textrm{sp}}^{\,*}(M_{1})
\,=\,
\lambda_{1}\,\widetilde{M}^{\prime}_{1},
\]
\[
\widetilde{p}_{\textrm{sp}}^{\,*}(M_{2})
\,=\,
\lambda_{2}\,\widetilde{M}^{\prime}_{2},
\]
and 
\[\widetilde{p}_{\textrm{sp
}}^{\,\ast}(H)\,=\, \widetilde{H}^{\prime},\]
where $ \widetilde{M}_{1}^{\prime}\, =\, c_{1}(\mathcal{O}_{\Gr(q_{1},\widetilde{E}_{1})}(1))-{\theta}_{1} c_{1}(\widetilde{L})$,\,  $\widetilde{M}_{2}^{\prime} \,=\, c_{1}(\mathcal{O}_{\Gr(q_{2},\widetilde{E}_{2})}(1)) - {\theta}_{2} c_{1}(\widetilde{L})$ and $ \widetilde{H}^{\prime}\,=\,c_{1}(\widetilde{L})$.
Therefore,
\[
\widetilde{p}_{\textrm{sp}}^{\,*}
\left(
M_{1}^{\,a-1}M_{2}^{\,b}H
\right)
\,=\,
\lambda_{1}^{\,a-1}\lambda_{2}^{\,b}\,
\widetilde{M}_{1}^{\prime\,a-1}
\widetilde{M}_{2}^{\prime \,b}
\widetilde{H},
\]
\[
\widetilde{p}_{\textrm{sp}}^{\,*}
\left(
M_{1}^{\,a}M_{2}^{\,b-1}H
\right)
\,=\,
\lambda_{1}^{\,a}\lambda_{2}^{\,b-1}\,
\widetilde{M}_{1}^{\prime \,a}
\widetilde{M}_{2}^{\prime \,b-1}
\widetilde{H},
\]
and
\[
\widetilde{p}_{\textrm{sp}}^{\,*}
\left(
M_{1}^{\,a}M_{2}^{\,b}
\right)
\,=\,
\lambda_{1}^{\,a}\lambda_{2}^{\,b}\,
\widetilde{M}_{1}^{\prime \,a}
\widetilde{M}_{2}^{\prime \,b}.
\]
In $N_{n_{q_{1}q_{2}}-1}(\widetilde{S})$ we have
\begin{equation*}
\widetilde{H}'
\,=\,
|\Gamma|\widetilde{H},\quad \widetilde{M}_{1}'
\,=\,
\widetilde{M}_{1},\quad \textrm{and}\quad \widetilde{M}_{2}'
\,=\,
\widetilde{M}_{2},
\end{equation*}
where $\widetilde{M}_{1}$, $\widetilde{M}_{2}$
and $\widetilde{H}$ are defined in $\eqref{m}$ and \eqref{h}. Moreover, in $N_{1}(\widetilde{S})$ we have
\[
\lambda_{1}^{\,a-1}\lambda_{2}^{\,b}\,
\widetilde{M}_{1}^{\,a-1}
\widetilde{M}_{2}^{\,b}
\widetilde{H}'
\,=\,
\lambda_{1}^{\,a-1}\lambda_{2}^{\,b}|\Gamma|\,
\widetilde{M}_{1}^{\,a-1}
\widetilde{M}_{2}^{\,b}
\widetilde{H},
\]
and
\[
\lambda_{1}^{\,a}\lambda_{2}^{\,b-1}\,
\widetilde{M}_{1}^{\prime \,a}
\widetilde{M}_{2}^{\prime \,b-1}
\widetilde{H}^{\prime}
\,=\,
\lambda_{1}^{\,a}\lambda_{2}^{\,b-1}|\Gamma|\,
\widetilde{M}_{1}^{\,a}
\widetilde{M}_{2}^{\,b-1}
\widetilde{H}.
\]
Thus all three classes
\[
M_{1}^{\,a-1}M_{2}^{\,b}H,
\quad
M_{1}^{\,a}M_{2}^{\,b-1}H,
\quad \textrm{and}\quad 
M_{1}^{\,a}M_{2}^{\,b}
\]
pull back to positive multiples of effective classes lying on the boundary of the Mori cone of $\widetilde{S}$. Now, using the same idea as in the proof of Theorem \ref{thm:mori-parabolic}, the proof of the theorem follows.
\end{proof}


\begin{thebibliography}{999999999}
\bibitem[Bi1]{Bi1}Biswas, I.:
Parabolic bundles as orbifold bundles.
\emph{Duke Math. J.} \textbf{88} (1997), 305--325.
\bibitem[Bi2]{Bi2} Biswas, I.: Connections on a parabolic principal bundle, II.  \emph{Canadian Math. Bull} \textbf{52} (2009), 175--185.
\bibitem[Br]{Br}
Brion, M.:
Lectures on the geometry of flag varieties. \url{https://doi.org/10.48550/arXiv.math/0410240}, (2004).


\bibitem[BBM1]{BBM1}
Bansal, A., Biswas, I., Majumder, S.:
On Positive Cones of Finite Normal Quotients of a Smooth Variety.
arXiv preprint, arXiv:2606.29803, 2026.


\bibitem[BBM2]{BBM2} Bansal, A., Biswas, I., Majumder, S.: Positive cones of the projectivization of a parabolic bundle over a curve. arXiv  preprint, arXiv:2506.17594, 2025.


\bibitem[BBN]{BBN}
Balaji, V., Biswas, I., Nagaraj, D. S.:
Ramified \(G\)-bundles as parabolic bundles.
\emph{J. Ramanujan Math. Soc.} \textbf{18} (2003), 123--138.

\bibitem[BL]{BL}
Biswas, I., Laytimi, F.:
Parabolic \(k\)-ample bundles.
\emph{Internat. J. Math.} \textbf{22} (2011), 1647--1660.

\bibitem[BP]{BP}
Biswas, I., Parameswaran, A. J.:
Nef cone of flag bundles over a curve.
\emph{Kyoto J. Math.} \textbf{54}  (2014), 353--366 .

\bibitem[BHP]{BHP}
Biswas, I., Hogadi, A., Parameswaran, A. J.:
Pseudo-effective cone of Grassmann bundles over a curve.
\emph{Geom. Dedicata} \textbf{172} (2014), 69--77.

\bibitem[BHNN]{BHNN}
Biswas, I., Hanumanthu, K., Nagaraj, D. S., Newstead, P. E.:
Seshadri constants and Grassmann bundles over curves.
\emph{Ann. Inst. Fourier} \textbf{70} (2020), 1477--1496.


\bibitem[DELV]{DELV}
Debarre, O., Ein, L., Lazarsfeld, R., Voisin, C.:
Pseudoeffective and nef classes on abelian varieties.
\emph{Compos. Math.} \textbf{147} (2011), 1793--1818.

\bibitem[Flt]{Fu}
Fulton, W.:
\emph{Intersection Theory}, 2nd edition,
Springer-Verlag, Berlin, 1998.

\bibitem[FL1]{FL1}
Fulger, M., Lehmann, B.:
Morphisms and faces of pseudo-effective cones.
\emph{Proc. London Math. Soc.} \textbf{112} (2016), 651--676.

\bibitem[FL2]{FL2}
Fulger, M., Lehmann, B.:
Positive cones of dual cycle classes.
\emph{Algebraic Geometry} \textbf{4} (2017), 1--28.
\bibitem[Flg]{Ful}
Fulger, M.:
The cones of effective cycles on projective bundles over curves.
\emph{Math. Z.} \textbf{269} (2011), 449--459.
\bibitem[KMR]{KMR}
Karmakar, R., Misra, S., Ray, N.:
Nef and pseudoeffective cones of the product of projective bundles over a curve.
\emph{Bull. Sci. Math.} \textbf{151} (2019), 1--12.
\bibitem[La]{La}
Lazarsfeld, R.:
\emph{Positivity in Algebraic Geometry I: Classical Setting},
Springer-Verlag, Berlin, 2004.

\bibitem[Mi]{Mi}
Miyaoka, Y.:
The Chern classes and Kodaira dimension of a minimal variety.
In: \emph{Algebraic Geometry}, Sendai, 1985,
\emph{Adv. Stud. Pure Math.} \textbf{10}  (1987), 449--476.
\bibitem[MS]{MS} Mehta, V.B., Seshadri, C.S.: Moduli of vector bundles on curves with parabolic structure. \textit{Math. Ann.} \textbf{248} (1980), 205--239.
\bibitem[MY]{MY}
Maruyama, M., Yokogawa, K.:
Moduli of parabolic stable sheaves.
\emph{Math. Ann.} \textbf{293} (1992), 77--99.
\bibitem[Se]{Se}
Seshadri, C. S.:
Moduli of vector bundles on curves with parabolic structures.
\emph{Bull. Amer. Math. Soc.} \textbf{83} (1977), 124--126.

\bibitem[SP]{SP}
The Stacks Project Authors:
\emph{Stacks Project}.
\url{https://stacks.math.columbia.edu} (2018).
\bibitem[MR]{MR}
Misra, S., Ray, N.:
Slope semistability and positive cones of Grassmann bundles.
\emph{J. Algebra Appl.} \textbf{23} (2024).
\bibitem[Yo]{Yo}
Yokogawa, K.:
Infinitesimal deformation of parabolic Higgs sheaves.
\emph{Int. J. Math.} \textbf{6} (1995), 125--148.
\end{thebibliography}
\end{document}